\documentclass[a4paper,12pt,final]{article}

\usepackage{amsmath,amsthm,amssymb}
\usepackage[a4paper,margin=1in]{geometry} 

\usepackage[utf8]{inputenc}
\usepackage[T1]{fontenc}
\usepackage{lmodern}

\usepackage{dsfont}

\usepackage{mathrsfs}

\usepackage{stmaryrd}

\usepackage{verbatim}

\usepackage[lowtilde]{url}

\usepackage{paralist}

\newtheorem{theorem}{Theorem}[section]
\newtheorem{proposition}[theorem]{Proposition}
\newtheorem{corollary}[theorem]{Corollary}
\newtheorem{lemma}[theorem]{Lemma}

\theoremstyle{definition}

\newtheorem{assumption}[theorem]{Assumption}
\newtheorem{example}[theorem]{Example}

\theoremstyle{definition}
\newtheorem{remark}[theorem]{Remark}

\swapnumbers
\numberwithin{equation}{section}

\newcommand\bE{\mathds{E}}

\newcommand\bN{\mathds{N}}
\newcommand\bP{\mathds{P}}
\newcommand\bR{\mathds{R}}

\newcommand{\one}{\mathds 1}

\newcommand\cB{\mathcal{B}}

\newcommand\cF{\mathcal{F}}
\newcommand\cG{\mathcal{G}}
\newcommand\cH{\mathcal{H}}

\newcommand\cM{\mathcal{M}}
\newcommand\cP{\mathcal{P}}

\newcommand\cO{\mathcal{O}}
				
\newcommand{\nrm}[1]{\ensuremath{\| #1 \|}}

\newcommand{\nnrm}[2]{\ensuremath{\| #1 \|_{#2}}}

\newcommand{\gnnrm}[2]{\ensuremath{\big\| #1 \big\|_{#2}}}

\newcommand{\sgnnrm}[2]{\ensuremath{\Big\| #1 \Big\|_{#2}}}

\renewcommand{\epsilon}{\ensuremath{\varepsilon}}

\renewcommand{\geq}{\ensuremath{\geqslant}}
\renewcommand{\leq}{\ensuremath{\leqslant}}

\newcommand{\dl}{\ensuremath{\mathrm{d}}}

\renewcommand{\arg}{\,\cdot\,}

\newcommand{\bo}{\ensuremath{\mathscr L}}

\newcommand{\nuc}{\ensuremath{\mathscr L_1}}

\newcommand{\hs}{\ensuremath{\mathscr L_2}}

\newcommand{\id}[1]{\ensuremath{\operatorname{id}}_{#1}}

\newcommand{\dom}{\ensuremath{\mathcal{O}}}

\newcommand{\dt}{\ensuremath{\Delta t}}

\allowdisplaybreaks

\title{
Weak convergence of finite element approximations of linear stochastic evolution equations with additive L\'{e}vy noise
\thanks{
Financial support from the Deutsche Forschungsgemeinschaft (DFG) within the Priority Program 1324 through a fellowship for the first author, grant SCHI 419/5-2 and Marsden Fund project number UOO1418 are gratefully acknowledged.
}
}
\author{Mih\'aly Kov\'acs, Felix Lindner and Ren\'e L.\ Schilling}
\date{}

\begin{document}
\allowdisplaybreaks
\maketitle

\begin{abstract}
    We present an abstract framework to study weak convergence of numerical approximations of linear stochastic partial differential equations driven by additive L\'evy noise. We first derive a representation formula for the error which we then apply to study space-time discretizations of the stochastic heat equation, a Volterra-type integro-differential equation, and the wave equation as examples.  For twice continuously differentiable test functions with bounded second derivative (with an additional condition on the second derivative for the wave equation) the weak rate of convergence is found to be twice the strong rate. The results extend earlier work by two of the authors as we consider general square-integrable infinite-dimensional L\'evy processes and do not require boundedness of the test functions and their first derivative. Furthermore, the present framework is applicable to both hyperbolic and parabolic equations, and even to stochastic Volterra integro-differential equations.
\end{abstract}
\\[\medskipamount]

\section{Introduction}
Let $H$ be a real separable Hilbert space and $(\Omega,\cF,(\cF_t)_{t\geq0},\bP)$ be a stochastic basis satisfiying the usual conditions, $L=(L(t))_{t\geq 0}$ be a square-integrable cylindrical L\'evy process in a real separable Hilbert space $U$ with respect to the stochastic basis $(\Omega,\cF,(\cF_t)_{t\geq0},\bP)$, taking values in a possibly larger Hilbert space $U_1\supset U$, and $B:U\to H$ is a bounded linear operator. Consider an $H$-valued stochastic convolution process
\begin{equation}\label{eq:X}
X(t)=E(t)X_0+\int_0^tE(t-s)B\,\dl L(s)
\end{equation}
where $(E(t))_{t\in [0,T]}$ is a family of bounded linear operators on $H$ and $X_0$ is an $\cF_0$-measurable $H$-valued random variable. Without loss of generality, all Hilbert spaces are assumed to be infinite-dimensional.
Important examples of such processes are weak solutions $(X(t))_{t\geq0}$ of certain stochastic partial differential equations (SPDEs, for short) driven by additive L\'{e}vy noise; these can be written as abstract It\^o stochastic differential equations
\begin{equation}\label{eq:mainEq}
\dl X(t)+AX(t)\,\dl t=B\,\dl L(t),\quad t\geq0;\quad X(0)=X_0,
\end{equation}
where $-A$ is the generator of a strongly continuous semigroup $(E(t))_{t\geq 0}$ on $H$. In particular, we consider the stochastic heat equation
\begin{equation}\label{eq:SHE}
\dl X(t)+\Lambda X(t)\dl t= \dl L(t),\quad t\geq0;\quad X(0)=X_0,
\end{equation}
and the stochastic wave equation, written as a first order system,
\begin{equation}\label{eq:SWE}
\begin{aligned}
\dl X_1(t) -X_2(t)\dl t&=0, &&t\geq0;\quad X_1(0)=X_{0,1},\\
\dl X_2(t)+\Lambda X_1(t)\dl t&= \dl L(t),&&t\geq0;\quad X_2(0)=X_{0,2}.
\end{aligned}
\end{equation}
In both cases $\Lambda:=-\Delta=-\sum_{j=1}^d\partial^2/\partial\xi_j^2$ is the Laplace operator on $L^2(\dom)$ where $\dom\subset\bR^d$ is a bounded domain.
For the precise abstract setup of these equations we refer to Sections \ref{sec:Heat} and \ref{sec:Wave}.
In general, however, we do not require that $(E(t))_{t\geq 0}$ enjoys the semigroup property so that the abstract framework can accommodate Volterra-type evolution equations as well, see Subsection \ref{subsec:sve} for details.

Consider an
approximation $\tilde X=(\tilde X(t))_{t\in [0,T]}$ of the process $(X(t))_{t\in[0,T]}$ given by
\begin{equation}\label{eq:Xtilde}
\tilde X(t)=\tilde E(t)X_0+\int_0^t\tilde E(t-s)B\,\dl L(s),
\end{equation}
where $(\tilde E(t))_{t\in[0,T]}$ is a family of bounded linear operators on $H$, which is again not necessarily (extendable to) an operator semigroup. For example, the family  $(\tilde E(t))_{t\in[0,T]}$ may be a time-interpolated solution operator family of a space-time discretized stochastic evolution problem, when $H$ is an $L^2$-space of some spatial domain $\dom$. We study the so-called weak error
\begin{equation}\label{eq:defWeakError}
e(T):=\bE\big(G(\tilde X(T))-G(X(T))\big)
\end{equation}
for suitable test functions $G:H\to \bR$. At the heart of the paper are the error representation formulae for $e(T)$, Theorem \ref{thm:errorRep} and Corollary \ref{cor:errorRep}. The proof of Theorem \ref{thm:errorRep} is
based on Kolmogorov's backward equation for the martingale $Y(t)=E(T)X_0+\int_0^tE(T-s)B\,\dl L(s)$, $t\in[0,T]$, which has the important property that $Y(T)=X(T)$. The introduction of such an auxiliary process $Y$ is well-known for equations with Gaussian noise and has been used by many authors in a weak error analysis, see, for example \cite{debussche_printems,KovLarLin11,KovLarLin13,KovPri14b} to mention just a few (compare also \cite{CJK14, DPJR12}). However, the extension of those arguments is not straightforward and the resulting error representation formula differs from the one in the Gaussian case in \cite{KovLarLin13}. One of the difficulties in the general L\'evy case (in contrast to the Gaussian case) is that there are no readily available, sufficiently general results on  Kolmogorov's backward equation to suit our analysis. We remedy this, at least for $Y$ as above, in Proposition \ref{prop:backwardKolm}. Another complication arises from the fact that we use tools from the theory of stochastic integration based on two different settings. One, where we integrate operator-valued processes w.r.t.\ a Hilbert space-valued L\'evy process, promoted in the monographs \cite[Chapter 8]{PesZab07}, \cite{Met82,MetPel80}, and another one where we integrate Hilbert space-valued integrands w.r.t.\ a Poisson random measure \cite{MRT13,Pre10}. The problem occurs because our setting for stochastic differential equations is based on the first approach while the proof of the error representation formula is based on an It\^o formula which appears in \cite[Theorem 3.6]{MRT13}; the latter form is well suited for our purposes, but it is formulated using the second approach for stochastic integration. Therefore, in the appendix we link the two stochastic integrals so that we can use the results from both theories.

Using the abstract error representation we study the weak error of  space-time discretizations for parabolic equations, such as the stochastic heat equation and a stochastic Volterra integro-differential equation, and a hyperbolic equation, the stochastic wave equation. As space discretization we employ a standard continuous finite element method. For the stochastic heat equation we use the backward Euler method, for the Volterra integro-differential equation the backward Euler method combined with a convolution quadrature, and for the stochastic wave equation an $I$-stable rational approximation of the exponential function, such as the Crank-Nicolson scheme, as time integrators. For all equations considered here, the Hilbert-Schmidt norm condition $$\nnrm{\Lambda^{(\beta-1/\rho)/2}Q^{1/2}}{\hs(L^2(\dom))}<\infty, \quad \beta>0,$$ determines the rate of convergence, where $\rho=1$ for the heat and wave equations and $\rho\in (1,2)$ for the Volterra equation. Here $U=L^2(\dom)$ and $Q\in\bo(L^2(\dom))$ is the covariance operator of $L$ as introduced in Section~\ref{sec:DrivingProcess}.

For the stochastic heat equation, we show in Theorem \ref{thm:SHE} that for twice continuously differentiable test functions with bounded second derivatives the rate of weak convergence is essentially twice that of strong convergence.  It is at least $O(h^{2\beta}+(\dt)^\beta)$, $\beta\in(0,1]$, modulo a logarithmic term, where $h$ and $\dt$ are the space- and time-discretization parameters, respectively. This extends the corresponding result from \cite{LinSch13}, where, in contrast to the present paper, the analysis is restricted to so-called impulsive cylindrical processes on $L^2(\dom)$ as driving noise. Moreover, there is a serious restriction on the jump size intensity measure in \cite[Section~6]{LinSch13}
admitting only processes of bounded variation (on finite time intervals).
Here, the only restriction we have on $L$ is that it is square-integrable, non-Gaussian and has mean zero. Furthermore, we also remove the boundedness assumption on the test functions and their first derivative.

In Subsection \ref{subsec:sve} we briefly discuss a stochastic Volterra-type integro-differential equation and obtain a weak rate of order at least $O(h^{2\beta}+(\dt)^{\rho\beta})$, $\beta\in(0,1/\rho]$, where $\rho$ depends on the order of the convolution kernel appearing in the equation.


For the stochastic wave equation the additional technical condition \eqref{eq:assgSWE} has to be imposed in order to prove that the weak order is twice the strong order. The order of weak convergence is found to be at least $O(h^{\min(r, 2\beta\frac{r}{r+1})}+(\dt)^{\min(1,2\beta\frac{p}{p+1})})$, see Theorem~\ref{thm:SWE}. Here $p$ and $r$ are the classical orders of the time-discretization and of the finite element method. We would like to point out that, while the extra condition \eqref{eq:assgSWE} on the second derivative on the test function is restrictive, it trivially holds for the important function $g(x)=\|x\|_{L^2(\dom)}^2$. We also give a sufficient condition \eqref{eq:weqii} for the above rate of weak convergence to hold which only involves the jump intensity measure of $L$ and $\Lambda$, with no further restriction on the test functions. Condition  \eqref{eq:weqii} is sufficient for $\nnrm{\Lambda^{(\beta-1)/2}Q^{1/2}}{\hs(L^2(\dom))}<\infty$ and, in special cases, it is even equivalent to it, see Example~\ref{ex:spec}.
Furthermore, as far as the authors know, there are no results available in the literature concerning weak approximation of hyperbolic stochastic partial differential equations driven by L\'evy noise.

Let us remark that weak error estimates for approximations of L\'{e}vy-driven stochastic ordinary differential equations have been considered by various authors, see, e.g.~\cite{JKMP05, MikZha11, PlaBru10, ProTal97} and the references therein. There also exists a series of papers on strong error estimates for approximations of SPDEs  driven by L\'{e}vy processes or Poisson random measures, see, for example \cite{Bar10, BarLan12a, BarLan12b, DunHauPro12, Hau08, HauMar06, Lan10} and compare also Remarks \ref{rem:strongErrorSHE} and \ref{rem:strongErrorSWE} below. However, to the best of our knowledge, the first steps in a \emph{weak} error analysis for L\'{e}vy-driven SPDEs have been done only recently in the already mentioned article \cite{LinSch13}.

The present paper
is organized as follows. In Section \ref{sec:Setting and preliminaries} we describe the abstract framework of the paper, introduce infinite-dimensional L\'evy processes with several examples and a framework for linear stochastic partial differential equations driven by additive L\'evy noise. Assumption \ref{ass:abstractSetting} summarizes the main assumptions for the general setting of the paper. In Section \ref{sec:An error representation formula} we state and prove two representation formulae, Theorem~\ref{thm:errorRep} and Corollary~\ref{cor:errorRep},  for $e(T)$ given by \eqref{eq:defWeakError}. The main ingredient in their proofs is Proposition~\ref{prop:backwardKolm} on Kolmogorov's backward equation. We use the representation formula from Corollary \ref{cor:errorRep} to establish weak convergence rates for a space-time discretization scheme first for parabolic equations in Section \ref{sec:Heat}, where we consider the stochastic heat equation (Subsection~\ref{subsec:she}) and, in less detail, a stochastic Volterra-type integro-differential equation (Subsection~\ref{subsec:sve}). Then, in Section \ref{sec:Wave}, we analyse the weak error of a numerical scheme for the stochastic wave equation. In the appendix we link stochastic integration with respect to Poisson random measures to integration with respect to infinite-dimensional L\'evy processes.

\section{Setting and preliminaries}
\label{sec:Setting and preliminaries}

Here we describe in detail our abstract setting and collect some background material from infinite-dimensional stochastic analysis.\\

\noindent
{\bf General notation.}
Let $(\cH,\langle\arg,\arg\rangle_{\cH})$ and $(\cG,\langle\arg,\arg\rangle_{\cG})$ be real, separable Hilbert spaces and denote by $\bo(\cH,\cG)$, $\nuc(\cH,\cG)$ and $\hs(\cH,\cG)$ the spaces of linear and bounded operators, nuclear operators and Hilbert-Schmidt operators from $\cH$ to $\cG$, respectively.  The corresonding norms are denoted by $\nnrm{\arg}{\bo(\cH,\cG)}$, $\nnrm{\arg}{\nuc(\cH,\cG)}$ and $\nnrm{\arg}{\hs(\cH,\cG)}$. If $\cH=\cG$, we write $\bo(\cH)$, $\nuc(\cH)$ and $\hs(\cH)$ instead of $\bo(\cH,\cH)$, $\nuc(\cH,\cH)$ and $\hs(\cH,\cH)$. Given a measure space $(M,\cM,\mu)$ and $1\leq p <\infty$, we denote by $L^p(M;\cH)=L^p(M,\cM,\mu;\cH)$ the space of all $\cM/\cB(\cH)$-measurable mappings $f:M\to\cH$ with finite norm $\nnrm{f}{L^p(M;\cH)}=(\bE\nnrm{f}{\cH}^p)^{1/p}$, where $\cB(\cH)$ denotes the Borel $\sigma$-algebra on the Hilbert space $\cH$. By $C^n(\cH,\bR)$ we denote the space of all $n$-times continuously Fr\'{e}chet differentiable functions $f:\cH\to\bR,\,x\mapsto f(x)$. By $C_{\text b}^n(\cH,\bR)$ we denote the subspace of functions from $C^n(\cH,\bR)$ which are bounded together with their derivatives. Identifying $\cH$ and $\bo(\cH,\bR)$ via the Riesz isomorphism, we consider for fixed $x\in\cH$ the first derivative  $f'(x)$ as an element of $\cH$. Similarly, the second derivative $f''(x)$ is considered as an element of $\bo(\cH)$.
We also write $f_x$ and $f_{xx}$ instead of $f'$ and $f''$.

\subsection{The driving L\'{e}vy process $L$}
\label{sec:DrivingProcess}

The process $L=(L(t))_{t\geq0}$ in Eq.~\eqref{eq:mainEq} is a  L\'{e}vy process with values in a real and separable Hilbert space $U_1$, defined on a filtered probability space $(\Omega,\cF,(\cF_t)_{t\geq0},\bP)$ satisfying the usual conditions (cf.~\cite{PesZab07}).
$L$ is $(\cF_t)$-adapted and for $t,h\geq0$ the increment $L(t+h)-L(t)$ is independent of $\cF_t$. We always consider a c\`{a}dl\`{a}g
(right continuous with left limits) modification of $L$, i.e., a modification such that $L(t)=\lim_{s\searrow t}L(s)$ for all $t\geq0$ and $L(t-):=\lim_{s\nearrow t}L(s)$ exists for all $t>0$, where the limits are pathwise limits in $U_1$.
Our standard reference for Hilbert space-valued L\'{e}vy processes is \cite{PesZab07}.

In order to keep the exposition simple, we assume that $L$ is square-integrable, i.e., $\bE\nnrm{L(t)}{U_1}^2<\infty$, and that the Gaussian part of $L$ vanishes. Moreover, we assume that $L$ has mean zero, i.e., $\bE L(t)=0$ in $U_1$.
Let $\nu$ be the jump intensity measure (L\'{e}vy measure) of $L$. Note that the jump intensity measure $\nu$ of a general L\'{e}vy process in $U_1$ satisfies $\nu(\{0\})=0$ and $\int_{U_1}\min(1,\nnrm{y}{U_1}^2)\nu(dy)<\infty$, cf.~\cite[Section 4]{PesZab07}. Due to our assumptions we have
\begin{equation}\label{eq:assnu}
\int_{U_1}\nnrm{y}{U_1}^2\nu(\dl y)<\infty,
\end{equation}
and the characteristic function of $L$ is given by
\begin{equation}\label{eq:charFunctionL}
\bE e^{i\langle x,L(t)\rangle_{U_1}}=\exp\Big\{-t\int_{U_1}\big(1-e^{i\langle x,y\rangle_{U_1}}+i\langle x,y\rangle_{U_1}\big)\nu(\dl y)\Big\},\quad t\geq0,\;x\in U_1.
\end{equation}
Conversely, any $U_1$-valued L\'{e}vy process $L$ satisfying \eqref{eq:assnu} and \eqref{eq:charFunctionL} is square-integrable, with mean zero and vanishing Gaussian part.


Let $Q_1\in\nuc(U_1)$ be the covariance operator of $L$.
It is determined by the jump intensity measure $\nu$ via
\begin{equation}\label{eq:Q1nu}
\langle Q_1x,y\rangle_{U_1}=\int_{U_1}\langle x,z\rangle_{U_1}\langle y,z\rangle_{U_1}\nu(\dl z),
\quad x,y\in U_1,
\end{equation}
see~\cite[Theorem 4.47]{PesZab07}.
Further, let
\begin{equation*}
(U_0,\langle\arg,\arg\rangle_{U_0}):=\big(Q_1^{1/2}(U_1),\langle Q_1^{-1/2}\arg,Q_1^{-1/2}\arg\rangle_{U_1}\big)
\end{equation*}
be the reproducing kernel Hilbert space of $L$, where $Q_1^{-1/2}$ denotes the pseudo-inverse of $Q_1^{1/2}$, see~\cite[Section 7]{PesZab07}. Recall that the operator $B$ in Eq.~\eqref{eq:mainEq} is defined on the Hilbert space $U$. We assume that
\begin{equation}\label{eq:U0subsetU}
U_0\subset U\subset U_1,
\end{equation}
and that the inclusions \eqref{eq:U0subsetU} define continuous embeddings. We denote the embedding of $U_0$ into $U$ by $J_0\in \bo(U_0,U)$ and set
\begin{equation}\label{eq:defQ}
Q:=J_0J_0^*\in\bo(U).
\end{equation}
The nonnegative and symmetric  operator $Q$ is the covariance operator of $L$ considered as a cylindrical process in $U$, cf.~Remark~\ref{rem:Q} below. As a consequence of Douglas' theorem as stated in \cite[Appendix A.4]{PesZab07}, compare also \cite[Corollary C.0.6]{PreRoeck07}, the reproducing kernel Hilbert space of $L$ has the alternative representation
\begin{equation*}
(U_0,\langle\arg,\arg\rangle_{U_0})=\big(Q^{1/2}(U),\langle Q^{-1/2}\arg,Q^{-1/2}\arg\rangle_U\big).
\end{equation*}

\begin{remark}\label{rem:Q}
Suppose w.l.o.g.\ that $U$ is dense in $U_1$, identify $U$ and $U^*$ via the Riesz isomorphism, and consider the Gelfand triple $U_1^*\subset U^*\equiv U\subset U_1$. Then it is not difficult to see that
\[\bE\langle L(t),x\rangle\langle L(t),y\rangle=t\langle Qx,y\rangle_U,\quad t\geq0,\;x,y\in U_1^*,\]
where $\langle\arg,\arg\rangle:U_1\times U_1^*\to\bR$ is the canonical dual pairing; compare \cite[Proposition 7.7]{PesZab07}.
The unique continuous extensions of the linear mappings $U_1^*\ni x\mapsto \langle L(t),x\rangle\in L^2(\bP)$, $t\geq0$, to the larger space $U^*$ determine a $2$-cylindrical $U$-process in the sense of \cite{MetPel80}, compare also \cite{AppRie10}, \cite{Rie12}, \cite{Rie14}.
\end{remark}

\begin{remark}\label{rem:approxL}
Unlike in the case of a mean-zero (cylindrical) $Q$-Wiener process in $U$, the covariance operators $Q\in\bo(U)$ and $Q_1\in\nuc(U_1)$ do \emph{not} determine the distribution of the L\'{e}vy process $L$, but the jump intensity measure $\nu$ does so according to \eqref{eq:charFunctionL}. Note that the law of a general L\'{e}vy process is determined by its characteristics (L\'{e}vy triplet), cf.~\cite[Definition~4.28]{PesZab07}, and that the characteristics of $L$ are $(-\int_{\{\nnrm{y}{U_1}\geq1\}}y\,\nu(\dl y),0,\nu)$. Nevertheless, the operator $Q$ in $\eqref{eq:defQ}$ will play an important role in our error analysis.
Let us briefly make the connection of our setting to the construction of a cylindrical $Q$-Wiener process in $U$ as described in \cite{DPZ92}, \cite{PreRoeck07}. To this end, let $(f_k)_{k\in\bN}$ be an orthonormal basis of $U_1$ consisting of eigenvectors of $Q_1$ with eigenvalues $(\lambda_k)_{k\in\bN}$ and consider the orthonormal basis $(e_k)_{k\in\bN}$ of $U_0$ given by $e_k:=\lambda_k^{1/2}f_k$. To simplify notation we suppose for the moment that all eigenvalues $\lambda_k$ of $Q_1$ are strictly positive. Then, compare \cite[Section~4.8]{PesZab07}, the real-valued L\'{e}vy processes $L_k=(L_k(t))_{t\geq0}$, $k\in\bN$, given by
\[L_k(t):=\lambda_k^{-1/2}\langle L(t),f_k\rangle_{U_1}\]
are uncorrelated, i.e., $\bE L_k(t)L_j(s)=0$ if $k\neq j$, they satisfy $\bE(L_k^2(t))=t$, and we have
\begin{equation}\label{eq:LPexpansion}
L(t)=\sum_{k\in\bN}L_k(t)e_k.
\end{equation}
The infinite sum in \eqref{eq:LPexpansion} converges for all finite $T>0$ in the space $\cM^2_T(U_1)$ of  c\`{a}dl\`{a}g square-integrable $U_1$-valued $(\cF_t)$-martingales $M=(M(t))_{t\in[0,T]}$ with norm $\nnrm{M}{\cM^2_T(U_1)}=(\bE\nnrm{M(T)}{U_1}^2)^{1/2}$. In contrast to the Gaussian case, where uncorrelated coordinates are always independent, the coordinate processes $L_k$, $k\in\bN$, are in general only uncorrelated but \emph{not} independent.

Conversely, suppose that we are given an arbitrary symmetric and nonnegative operator $Q\in\bo(U)$, an orthonormal basis $(e_k)_{k\in\bN}$ of $U_0=Q^{1/2}(U)$, and a family $L_k$, $k\in\bN$, of real-valued L\'{e}vy processes on $(\Omega,\cF,(\cF_t)_{t\geq0},\bP)$ that satisfy the following conditions:
\begin{itemize}
\item Each $L_k$ is $(\mathcal F_t)$-adapted and for $t,h\geq0$ the increment $L_k(t+h)-L_k(t)$ is independent of $\cF_t$;
\item each $L_k$ is square-integrable with $\bE L_k(t)=0$ and $\bE(L_k^2(t))=t$;
\item the processes $L_k$, $k\in\bN$, are uncorrelated;
\item for all $n\in\bN$ the $\bR^n$-valued process $((L_1(t),\ldots,L_n(t))^\top)_{t\geq 0}$ is a L\'{e}vy process;
\item the Gaussian part of each $L_k$ is zero.
\end{itemize}
Then, if $U_1$ is a Hilbert space containing $U$ such that the natural embedding of $U_0=Q^{1/2}(U)$ into $U_1$ is Hilbert-Schmidt, the infinite sum in \eqref{eq:LPexpansion} converges in $\cM^2_T(U_1)$ and defines a L\'{e}vy process $L$ with reproducing kernel Hilbert space $U_0$ that fits into our setting.
\end{remark}

We end this subsection with some examples of L\'{e}vy processes $L$. We suppose that all processes are defined relative to the stochastic basis $(\Omega,\cF,(\cF_t)_{t\geq0},\bP)$ and that their increments on time intervals $[t,t+h]$ are independent of $\cF_t$.

\begin{example}(\emph{Subordinate cylindrical $\widetilde Q$-Wiener process})
Let $W=(W(t))_{t\geq0}$ be a cylindrical $\widetilde Q$-Wiener process in $U$ in the sense of \cite[Section 2.5.1]{PreRoeck07}, where $\widetilde Q\in\bo(U)$ is a given nonnegative and symmetric operator. Assume that $W$ takes values in a possibly larger Hilbert space $U_1\supset U$ such that the natural embedding of $U$ into $U_1$ is continuous. Let $\widetilde Q_1\in\nuc(U_1)$ be the covariance operator of $W$ considered as a Wiener process in $U_1$, i.e., $\bE\langle W(t),x\rangle_{U_1}\langle W(s),y\rangle_{U_1}=\min(s,t)\langle \widetilde Q_1x,y\rangle_{U_1}$ for $x,\,y\in U_1$, $s,t\geq0$. Let $Z=(Z(t))_{t\geq0}$ be a subordinator, i.e., a real-valued increasing L\'{e}vy process in the sense of \cite[Definition 21.4]{Sato13}, \cite{SSV10}. Assume that $W$ and $Z$ are independent, that the drift of $Z$ is zero, and that the jump intensity measure  $\rho$ of $Z$ satisfies
\begin{equation}\label{eq:measureSubordinator}
\int_0^\infty s\,\rho(\dl s)<\infty.
\end{equation}
The latter is equivalent to assuming that $Z$ has first moments, $\bE|Z(t)|<\infty$. According to \cite[Remark~21.6]{Sato13}, the Laplace tranform of $Z(t)$ is given by
\begin{equation}\label{eq:LaplaceSubordinator}
\bE(e^{-r Z(t)})=\exp\big(-t\int_0^\infty(1-e^{-rs})\rho(\dl s)\big),\quad r\geq0.
\end{equation}
In this situation, subordinate cylindrical Brownian motion
\begin{equation*}\label{eq:subWP}
L(t):=W(Z(t)),\quad t\geq0,
\end{equation*}
defines a $U_1$-valued L\'{e}vy process $L=(L(t))_{t\geq0}$ that fits into the general framework described above. Indeed, $L$ has stationary and independent increments. Moreover, the independence of $W$ and $Z$, the identity $\bE e^{i\langle x,W(s)\rangle_{U_1}}=e^{-s\frac12 \langle \widetilde Q_1x,x\rangle_{U_1}}$, Eq.~\eqref{eq:LaplaceSubordinator} and the symmetry of the distribution $\bP_{W(1)}=N(0,Q_1)$ imply that characteristic function of $L(t)$ is given by
\begin{align*}
\bE e^{i\langle x,L(t)\rangle_{U_1}}
&=\int_0^\infty e^{-s\frac12 \langle\widetilde Q_1x,x\rangle_{U_1}}\,\bP_{Z(t)}(\dl s)\\
&=\exp\Big[-t\int_0^\infty(1-e^{-\frac12\langle\widetilde Q_1x,x\rangle_{U_1}s})\rho(\dl s)\Big]\\
&=\exp\Big[-t\int_0^\infty\int_{U_1}(1-e^{i\langle x,\sqrt s y\rangle_{U_1}}+i\langle x,\sqrt s y\rangle_{U_1})\bP_{W(1)}(\dl y)\rho(\dl s)\Big].
\end{align*}
As a consequence, \eqref{eq:charFunctionL} holds with
\begin{equation}\label{eq:subWPnu}
\nu=(\bP_{W(1)}\otimes\rho)\circ\kappa^{-1},
\end{equation}
where $\kappa:U_1\times(0,\infty)\to U_1$ is defined by $\kappa(y,s)=\sqrt s y$; compare \cite[Lemma 4.8]{Rie14}. (Note that, by the scaling property of $W$, \eqref{eq:subWPnu} is equivalent to the standard formula $\nu=\int_0^\infty\bP_{W(s)}\,\rho(\dl s)$, where the measure-valued integral is defined in a weak sense, cf.~\cite[Section~30]{Sato13}). Moreover, \eqref{eq:assnu} holds due to \eqref{eq:measureSubordinator} as we have the equality
$\int_{U_1}\nnrm{y}{U_1}^2\nu(\dl y)=\int_0^\infty s\,\rho(\dl s)\,\bE(\nnrm{W(1)}{U_1}^2)$ according to \eqref{eq:subWPnu}.
It follows that $L$ is a $U_1$-valued, square-integrable, mean-zero L\'{e}vy process with vanishing Gaussian part. It is also not difficult to show that the covariance operators $Q_1\in\nuc(U_1)$ and $Q\in\bo(U)$ of $L$ in \eqref{eq:Q1nu} and \eqref{eq:defQ} are given by
$Q_1=\int_0^\infty s\,\rho(\dl s)\,\widetilde Q_1$ and $Q=\int_0^\infty s\,\rho(\dl s)\,\widetilde Q$.
Subordinate cylindrical Wiener processes have been considered, e.g., in \cite{BrzZab10}.
\end{example}

\begin{example}(\emph{Independent one-dimensional L\'{e}vy processes})\label{ex:1dlevy}
Let $Q\in\bo(U)$ be symmetric, nonnegative and let $(e_k)_{k\in\bN}$ be an orthonormal basis of $U_0:=Q^{1/2}(U)\subset U$. Let $L_k=(L(t))_{k\in\bN}$, $k\in\bN$, be independent real-valued square-integrable L\'{e}vy processes with vanishing Gaussian part and $\bE L_k(t)=0$, $\bE(L_k^2(t))=t$. Let $U_1\supset U$ be another Hilbert space such that the natural embedding of $U_0$ into $U_1$ is a Hilbert-Schmidt operator. Then, the series \eqref{eq:LPexpansion} converges for all $T\in(0,\infty)$ in the space $\cM^2_T(U_1)$ and defines a L\'{e}vy process $L=(L(t))_{t\geq0}$ satisfiying \eqref{eq:assnu} and \eqref{eq:charFunctionL} with jump intensity measure
\[\nu=\sum_{k\in\bN}\nu_k\circ\pi_k^{-1},\]
where $\nu_k$ is the L\'{e}vy measure of $L_k$ and $\pi_k:\bR\to U_1$ is defined by $\pi_k(\xi):=\xi e_k$; compare \cite[Section 4.8.1]{PesZab07}.
\end{example}

\begin{example}(\emph{Impulsive cylindrical process})
Let $\mu$ be a L\'{e}vy measure on $\bR$ such that $\int_\bR\sigma^2\mu(\dl\sigma)<\infty$. Let $\cO\subset\bR^d$ be a bounded domain and $Z=(Z(t))_{t\geq0}$ an impulsive cylindrical process on $U:=L^2(\cO)=L^2(\cO,\cB(\cO),\lambda^d)$ with jump size intensity $\mu$ in the sense of \cite[Definition 7.23]{PesZab07}. Here, $\lambda^d$ denotes $d$-dimensional Lebesgue measure.  The process $Z$ is a measure-valued process defined, informally, by $Z(t,\dl\xi)=\int_0^t\int_\bR\sigma\hat\pi(\dl s,\dl\xi,\dl \sigma)$, where $\hat\pi$ is a compensated Poisson random measure on $[0,\infty)\times\cO\times\bR$ with reference measure $\lambda^1\otimes\lambda^d\otimes\mu$; see \cite[Section 7.2]{PesZab07} for details. Let $\widetilde Q\in\bo(U)$ be symmetric and nonnegative, $(b_k)_{k\in\bN}$ an orthonormal basis of $U$, and $U_1\supset U$ a Hilbert space such that the natural embedding of $U_0=\tilde Q^{1/2}(U)\subset U$ into $U_1$ is Hilbert-Schmidt. Then the series
\begin{equation}\label{eq:imulsiveProcess}
L(t):=\widetilde Q^{\frac12} Z(t):=\sum_{k\in\bN}\int_0^t\int_\cO\int_\bR\sigma b_k(\xi)\hat\pi(\dl s,\dl\xi,\dl \sigma)\widetilde Q^{\frac12}b_k,\quad t\geq0,
\end{equation}
converges for all $T\in(0,\infty)$ in $\cM^2_T(U_1)$ and defines a L\'{e}vy process that fits into our general framework with $Q=\int_\bR\sigma^2\mu(\dl\sigma)\widetilde Q$ and
\[\nu=(\lambda^d\otimes\mu)\circ \phi^{-1},\]
where $\phi\in L^2(\cO\times\bR,\lambda^d\otimes\mu;U_1)$ is defined by $\phi(\xi,\sigma)=\sum_{n\in\bN}\sigma b_k(\xi)\widetilde Q^{\frac12}b_k$ (convergence in $L^2(\cO\times\bR,\lambda^d\otimes\mu;U_1)$).
In \cite{LinSch13} we considered the weak approximation of the stochastic heat equation driven by an impulsive process of the form \eqref{eq:imulsiveProcess}. The results in Section~\ref{subsec:she} of the present article improve the results of \cite{LinSch13} in several aspects.
\end{example}

\subsection{Linear stochastic evolution equations with additive noise}
\label{sec:LSEE}

We are mainly interested in equations of the type \eqref{eq:mainEq}, where $A:D(A)\subset H\to H$ is an unbounded linear operator such that $-A$ is the generator of a strongly continuous semigroup $(E(t))_{t\geq0}\subset\bo(H)$, $B\in\bo(U,H)$, $L=(L(t))_{t\geq0}$ is a square-integrable L\'{e}vy process with reproducing kernel Hilbert space $U_0\subset U$ as described in Subsection~\ref{sec:DrivingProcess}, and $X_0\in L^2(\Omega,\cF_0,\bP;H)$. It is well known that if
\begin{equation}\label{eq:assEB1}
\int_0^T \nnrm{E(t)B}{\hs(U_0,H)}^2\dl t <\infty
\end{equation}
for some (and hence for all) $T>0$, then there exists a unique weak solution $X=(X(t))_{t\geq0}$ to \eqref{eq:mainEq} which is given by the variation-of-constants formula \eqref{eq:X}, see, e.g., \cite[Chapter~9]{PesZab07}.
Similarly, if $(\tilde E(t))_{t\in[0,T]}\subset\bo(H)$
is given by some approximation scheme such that $t\mapsto\tilde E(t)B$ is a measurable mapping from $[0,T]$ to $\hs(U_0,H)$, then the condition
\begin{equation}\label{eq:assEBtilde1}
\int_0^T \nnrm{\tilde E(t)B}{\hs(U_0,H)}^2\dl t <\infty
\end{equation}
ensures that the approximation $\tilde X=(\tilde X(t))_{t\in[0,T]}$ of $(X(t))_{t\in[0,T]}$ in \eqref{eq:Xtilde} exists as a square-integrable $H$-valued process. We refer to \cite[Chapter~8]{PesZab07} for details on the construction and properties of the stochastic integral w.r.t.\ Hilbert space-valued L\'{e}vy processes.

It turns out that our general error representation formula for the weak error $e(T)$ in \eqref{eq:defWeakError} does not require the semigroup property of the strongly continuous family of operators $(E(t))_{t\geq0}$. This paves the way for analysing a more general class of L\'{e}vy-driven linear stochastic evolution equations, including for example stochastic Volterra-type equations as considered in \cite{KovPri14a}, \cite{KovPri14b} for the Gaussian case, see Subsection \ref{subsec:sve} for an example.  For such equations, the weak solution still has the form  \eqref{eq:X} but the solution operator family $(E(t))_{t\geq0}\subset\bo(H)$ is not a semigroup anymore.  Therefore, we weaken our abstract assumptions and summarize them as follows.

\begin{assumption} \label{ass:abstractSetting}
We will use the following assumptions:
\begin{compactenum}[(i)]
\item $H$, $U$ and $U_1$ are real and separable Hilbert spaces;
\item
$L=(L(t))_{t\geq0}$ is a $U_1$-valued L\'{e}vy process on $(\Omega,\cF,(\cF_t)_{t\geq0},\bP)$ with zero mean and finite second moments and reproducing kernel Hilbert space $U_0$ such that $U_0\subset U\subset U_1$ as described in Subsection~\ref{sec:DrivingProcess};
\item
$X_0\in L^2(\Omega,\cF_0,\bP;H)$;
\item
$B\in\bo(U,H)$ and $(E(t))_{t\in[0,T]}\subset \bo(H)$ is a strongly continuous family of linear operators such that \eqref{eq:assEB1} holds;
\item
for all $\epsilon>0$ there exists $\Phi_\epsilon\in\hs(U_0,H)$ and $C_\epsilon>0$ such that
\[\nnrm{E(t)Bx}{H}\leq \nnrm{\Phi_\epsilon x}H,\quad (t,x)\in[\epsilon,T]\times U_0;\]
\item
$(\tilde E(t))_{t\in[0,T]}\subset\bo(H)$ is a family of linear operators such that $t\mapsto\tilde E(t)B$ is a measurable mapping from $[0,T]$ to $\hs(U_0,H)$ and \eqref{eq:assEBtilde1} holds;
\item
$X=(X(t))_{t\in[0,T]}$ and $\tilde X=(\tilde X(t))_{t\in[0,T]}$ are $H$-valued stochastic processes given by \eqref{eq:X} and \eqref{eq:Xtilde}.
\end{compactenum}
\end{assumption}

\begin{remark}
If $(E(t))_{t\geq0}$ is an operator semigroup, then \ref{ass:abstractSetting}(v) is a consequence of \ref{ass:abstractSetting}(iv). Indeed, if \eqref{eq:assEB1} holds then $E(t_0)B\in \hs(U_0,H)$ for some $t_0\in (0,\epsilon)$ and hence
\[\nnrm{E(t)Bx}{H}\le \|E(t_0)Bx\|_H\sup_{t\in [t_0,T]}\|E(t-t_0)\|_{\bo(H)},\quad(t,x)\in[\epsilon,T]\times U_0.\]
Hence, one may take $\Phi_{\epsilon}:=cE(t_0)B$ where $c:=\sup_{t\in [t_0,T]}\|E(t-t_0)\|_{\bo(H)}$.
\end{remark}

To fix notation, let us briefly recall the It\^{o} isometry for stochastic integrals w.r.t.\ $L$. It has the same form as the It\^{o} isometry for stochastic integrals w.r.t.\ Hilbert space-valued Wiener processes.
We set $\Omega_T:=\Omega\times[0,T]$ and $\bP_T:=\bP\otimes\lambda$, where $\lambda$ is Lebesgue measure on $[0,T]$. The predictable $\sigma$-algebra on $\Omega_T$ w.r.t.\ $(\cF_t)_{t\in[0,T]}$ is denoted by $\cP_T$. For operator-valued processes $\Phi=(\Phi(t))_{t\in[0,T]}$ in
\[L^2(\Omega_T,\bP_T;\hs(U_0,H)):=L^2(\Omega_T,\cP_T,\bP_T;\hs(U_0,H)),\]
we have
\begin{equation}\label{eq:ItoIsom}
\bE\Big(\gnnrm{\int_0^t\Phi(s)\dl L(s)}H^2\Big)=\bE\int_0^t\nnrm{\Phi(s)}{\hs(U_0,H)}^2\dl s,\quad t\in[0,T],
\end{equation}
and the integral process $(\int_0^t\Phi(s)\dl L(s))_{t\in[0,T]}$ belongs to the space $\cM^2_T(H)$ of  c\`{a}dl\`{a}g square-integrable $H$-valued $(\cF_t)$-martingales. The usual norm in $\cM^2_T(H)$ is defined by $\nnrm{M}{\cM^2_T(H)}=(\bE\nnrm{M(T)}{H}^2)^{1/2}$, $M=(M(t))_{t\in[0,T]}\in\cM^2_T(H)$. Note, however, that the integral processes given by the stochastic integrals in \eqref{eq:X} and \eqref{eq:Xtilde} are in general \emph{not} martingales since the (deterministic) operator-valued integrands also depend on $t$.

We also recall
the definition and some properties of Hilbert-Schmidt operators,
cf.~\cite[Chapter~6]{Wei80}. Let $\cH$ and $\cG$ be real and separable Hilbert spaces. A linear and bounded operator $C\in\bo(\cH,\cG)$ belongs to the space $\hs(\cH,\cG)$ of Hilbert-Schmidt operators if
\[\nnrm{C}{\hs(\cH,\cG)}:=\Big(\sum_{k\in\bN}\nnrm{Ch_k}{\cG}^2\Big)^{1/2}<\infty\]
for some (and hence for every) orthonormal basis $(h_k)_{k\in\bN}$ of $\cH$. If $C\in\bo(\cH,\cG)$ and $C^*\in\bo(\cG,\cH)$ is the adjoint operator, then $C\in\hs(\cH,\cG)$ if and only if $C^*\in\hs(\cG,\cH)$ and one has
\begin{equation}\label{eq:propHSadjoint}
\nnrm{C}{\hs(\cH,\cG)}=\nnrm{C^*}{\hs(\cG,\cH)}.
\end{equation}
Also, if $C\in\hs(\cH,\cG)$, $D\in\bo(\cH)$ and $F\in\bo(\cG)$, then obviously $CD\in\hs(\cH,\cG)$, $FC\in\hs(\cH,\cG)$ and
\begin{equation}\label{eq:estHS}
\nnrm{CD}{\hs(\cH,\cG)}\leq\nnrm{C}{\hs(\cH,\cG)}\nnrm{D}{\bo(\cH)},\quad
\nnrm{FC}{\hs(\cH,\cG)}\leq\nnrm{F}{\bo(\cG)}\nnrm{C}{\hs(\cH,\cG)}.
\end{equation}
In particular, in our setting we have $\bo(U_1,H)\subset\hs(U_0,H)$ since
\[\nnrm{C}{\hs(U_0,H)}=\nnrm{CQ_1^{1/2}}{\hs(U_1,H)}\leq\nnrm{C}{\bo(U_1,H)}\nnrm{Q_1^{1/2}}{\hs(U_1)}\]
for all $C\in\bo(U_1,H)$ and $\nnrm{Q_1^{1/2}}{\hs(U_1)}=\text{Tr\,}Q_1=\nnrm{Q_1}{\nuc{(U_1)}}<\infty$.

\section{An error representation formula}
\label{sec:An error representation formula}

In this section, we state and prove a general representation formula for the weak approximation error $e(T)$ in \eqref{eq:defWeakError} within the abstract setting described above.

\subsection{Formulation of the result}

For the formulation and the proof of the error representation formula, we introduce auxiliary drift-free It\^{o} processes $Y=(Y(t))_{t\in[0,T]}$ and $\tilde Y=(\tilde Y(t))_{t\in[0,T]}$ such that
\begin{equation*}
X(T)=Y(T),\quad \tilde X(T)=\tilde Y(T).
\end{equation*}
The processes $Y$ and $\tilde Y$ are constructed by applying to $X$ and $\tilde X$ the deterministic operator-valued processes $(E(T-t))_{t\in[0,T]}$ and $(\tilde E(T-t))_{t\in[0,T]}$. That is, we set
\begin{equation}\label{eq:defY}
Y(t):=E(T)X_0+\int_0^tE(T-s)B\,\dl L(s), \quad t\in[0,T],
\end{equation}
and
\begin{equation}\label{eq:defYtilde}
\tilde Y(t):=\tilde E(T)X_0+\int_0^t\tilde E(T-s)B\,\dl L(s), \quad t\in[0,T].
\end{equation}

Moreover, we consider the auxiliary problem
\[\dl Z(t)=E(T-t)B\,\dl L(t),\quad t\in[\tau,T];\qquad Z(\tau)=\xi,\]
where $\tau\in[0,T)$ and $\xi$ is an $H$-valued $\cF_\tau$-measurable random variable. Its solution is given by
\begin{equation}\label{eq:defZttauxi}
Z(t,\tau,\xi):=\xi+\int_{\tau}^t E(T-s)B\,\dl L(s),\quad t\in[\tau,T],
\end{equation}
and we use it to define for $G\in C^2(H,\bR)$ with $\sup_{x\in H}\nnrm{G''(x)}{\bo(H)}<\infty$ a function $u:[0,T]\times H\to\bR$ by
\begin{equation}\label{eq:defuxt}
u(t,x):=\bE\,G(Z(T,t,x)),\quad (t,x)\in [0,T]\times H.
\end{equation}
Note that the boundedness of $G''$ implies quadratic and linear growth of $G$ and $G'$, respectively. That is,
\begin{equation}\label{eq:growthGG'}
|G(x)|\leq C(1+\nnrm{x}H^2),\quad G'(x)\leq C(1+\nnrm{x}H)
\end{equation}
for all $x\in H$ and a constant $C\in(0,\infty)$ that does not depend on $x$.
It is also not difficult to see that $u$
is twice Fr\'{e}chet differentiable w.r.t. $x$ and we have
\begin{equation}\label{eq:u_xu_xx}
u_x(t,x)=\bE\,G'(Z(T,t,x)),\quad u_{xx}(t,x)=\bE\,G''(Z(T,t,x)).
\end{equation}
All expectations appearing in \eqref{eq:defuxt} and \eqref{eq:u_xu_xx} make sense due to \eqref{eq:growthGG'}, the assumption $\sup_{x\in H}\nnrm{G''(x)}{\bo(H)}<\infty$, \eqref{eq:assEB1} and It\^{o}'s isometry \eqref{eq:ItoIsom}.

Before stating the representation formula, we show in the following lemma how operators in $\hs(U_0,H)$ can be identified with functions in
\begin{equation*}
L^2(U_1,\nu;H):=L^2(U_1,\mathcal B(U_1),\nu;H)
\end{equation*}
and how processes in $L^2(\Omega_T,\bP_T;\hs(U_0,H))$ can be identified with elements in
\begin{equation*}
L^2(\Omega_T\times U_1,\bP_T\otimes\nu;H):=L^2(\Omega_T\times U_1,\cP_T\otimes \cB(U_1),\bP_T\otimes\nu;H).
\end{equation*}
These identifications will be used implicitly throughout this article, see Remark~\ref{not:Phix} below. They also lead to a generic identification of integrals w.r.t.\ (cylindrical) Hilbert space-valued L\'{e}vy processes of jump type and integrals w.r.t.\ the associated Poisson random measures, cf.\ Appendix~\ref{sec:PRM+SI}.


\begin{lemma}\label{lem:integrandIsomorphism}
Let $(f_k)_{k\in\bN}\subset U_0$ be an orthonormal basis of $U_1$ consisting of eigenvectors of the covariance operator $Q_1\in\nuc(U_1)$ of $L$ and let $(\lambda_k)_{k\in\bN}\subset[0,\infty)$ be the corresponding sequence of eigenvalues

\textbf{(i)}
Given $\Phi\in\hs(U_0,H)$, the series
\[\iota(\Phi):=\sum_{k\in\bN,\lambda_k\neq 0}\langle\arg,f_k\rangle_{U_1}\Phi f_k\]
converges in $L^2(U_1,\nu; H)$.

The linear mapping
\[\iota:\hs(U_0,H)\to L^2(U_1,\nu;H),\;\Phi\mapsto \iota(\Phi)\]
is an isometric embedding.

\textbf{(ii)}
Given $\Phi\in L^2(\Omega_T,\bP_T;\hs(U_0,H))$, the series
\[\kappa(\Phi):=\sum_{k\in\bN,\lambda_k\neq0}\langle\arg,f_k\rangle_{U_1}\Phi(\arg) f_k\]
converges in $L^2(\Omega_T\times U_1,\bP_T\otimes\nu;H)$.
The linear mapping
\[\kappa:L^2(\Omega_T,\bP_T;\hs(U_0,H))\to L^2(\Omega_T\times U_1,\bP_T\otimes\nu;H),\;\Phi\mapsto \kappa(\Phi)\]
is an isometric embedding.  For $\bP_T$-almost all $(\omega,t)\in\Omega_T$ we have
$\kappa(\Phi)(\omega,t,\arg)=\iota(\Phi(\omega,t))$
in $L^2(U_1,\nu;H)$, where $\iota$ is the embedding from \textup{(i)}.
\end{lemma}

\begin{proof}
\textbf{(i)}
W.l.o.g.\ all eigenvalues $\lambda_k$ are strictly positive. Let $(e_k)_{k\in\bN}$ be the orthonormal basis of $U_0$ given by $e_k:=\lambda_k^{1/2}f_k$. For $m,n\in\bN$ with $m\leq n$ we have
\begin{align*}
\sgnnrm{\sum_{k=m}^n\langle\arg,f_k\rangle_{U_1}\Phi f_k}{L^2(U_1,\nu;H)}^2
&=\int_{U_1}\sgnnrm{\sum_{k=m}^n\langle x,f_k\rangle_{U_1}\Phi f_k}{H}^2\nu(\dl x)\\
&= \sum_{j,k=m}^n\lambda_j^{-1/2}\lambda_k^{-1/2}\int_{U_1}\langle x,f_j\rangle_{U_1}\langle x,f_k\rangle_{U_1}\nu(\dl x)\,\langle\Phi e_j,\Phi e_k\rangle_H\\
&=\sum_{k=m}^n\nnrm{\Phi e_k}H^2;
\end{align*}
in the last step we used \eqref{eq:Q1nu}. Since $\sum_{k\in\bN}\nnrm{\Phi e_k}H^2=\nnrm{\Phi}{\hs(U_0,H)}^2<\infty$, this shows that the partial sums $\sum_{k=1}^n\langle\arg,f_k\rangle_{U_1}\Phi f_k$, $n\in\bN$, are a Cauchy sequence in $L^2(U_1,\nu;H)$ and
\[\sgnnrm{\sum_{k=1}^\infty\langle\arg,f_k\rangle_{U_1}\Phi f_k}{L^2(U_1,\nu;H)}=\nnrm{\Phi}{\hs(U_0,H)}.\]

\textbf{(ii)} The first two assertions can be shown as in the proof of (i). The last assertion is due the fact that the iterated integral
\[\int_\Omega\int_0^T\int_{U_1}\nnrm{\iota(\Phi(\omega,t))(x)-\kappa(\Phi)(\omega,t,x)}H^2\,\nu(\dl x)\,\dl t\,\bP(\dl\omega)\]
equals zero, which follows from an approximation argument.
\end{proof}

\begin{remark}\label{not:Phix}
From now on we will
identify operators $\Phi\in\hs(U_0,H)$ with the corresponding mappings $\iota(\Phi)\in L^2(U_1,\nu;H)$ and write
\[\Phi x=\iota(\Phi)(x),\quad x\in U_1.\]
Analogously, we identify processes $\Phi\in L^2(\Omega_T,\bP_T;\hs(U_0,H))$ with the corresponding mappings $\kappa(\Phi)\in L^2(\Omega_T\times U_1,\bP_T\otimes\nu;H)$ and write
\[\Phi(\omega,t)x=\kappa(\Phi)(\omega,t,x),\quad (\omega,t,x)\in \Omega_T\times U_1.\]
For processes $\Phi\in L^2(\Omega_T,\bP_T;\hs(U_0,H))$ both identifications 
are compatible $\bP\otimes\lambda$-almost everywhere on $\Omega_T$ in the sense that we have $\kappa(\Phi)(\omega,t,\arg)=\iota(\Phi(\omega,t))$
in $L^2(U_1,\nu;H)$ for $\bP\otimes\lambda$-almost all $(\omega,t)\in\Omega_T$.
\end{remark}

Here is the main result of this section.
\begin{theorem}\label{thm:errorRep}
Let Assumption~\ref{ass:abstractSetting} hold and $G\in C^2(H,\bR)$ with $\sup_{x\in H}\nnrm{G''(x)}{\bo(H)}<\infty$.
Then, for the process $(\tilde Y(t))_{t\in[0,T]}$ from \eqref{eq:defYtilde} and the function $u:[0,T]\times H\to\bR$ from \eqref{eq:defuxt} it holds that
\begin{equation}\label{eq:errorRep1}
\begin{aligned}
\bE\int_0^T\int_{U_1}&\Big|u\big(t,\tilde Y(t)+\tilde E(T-t) By\big)-u\big(t,\tilde Y(t)+E(T-t)By\big)\\
&-\big\langle u_x(t,\tilde Y(t)),\big(\tilde E(T-t) B-E(T-t)B\big)y\big\rangle_H\Big|\,\nu(\dl y)\,\dl t\;<\;\infty.
\end{aligned}
\end{equation}
The weak error $e(T)$ in \eqref{eq:defWeakError} has the representation
\begin{equation}\label{eq:errorRep2}
\begin{aligned}
e(T)&=\bE\big\{u(0,\tilde E(T)X_0)-u(0,E(T)X_0)\big\}\\
&\quad +\bE\int_0^T\int_{U_1}\Big\{u\big(t,\tilde Y(t)+\tilde E(T-t) By\big)-u\big(t,\tilde Y(t)+E(T-t)By\big)\\
&\hspace{3.5cm}-\big\langle u_x(t,\tilde Y(t)),\big(\tilde E(T-t) B-E(T-t)B\big)y\big\rangle_H\Big\}\,\nu(\dl y)\,\dl t.
\end{aligned}
\end{equation}
\end{theorem}

\begin{remark}
The terms $E(T-t)By$ and $\tilde E(T-t) By$ appearing in \eqref{eq:errorRep1} and \eqref{eq:errorRep2} are defined for $\lambda\otimes\nu$-almost all $(t,y)\in[0,T]\times U_1$. This follows from \eqref{eq:assEB1}, \eqref{eq:assEBtilde1}, Lemma~\ref{lem:integrandIsomorphism} and Remark~\ref{not:Phix},
\end{remark}

We will prove Theorem~\ref{thm:errorRep} in the next subsection. Let us briefly record an alternative representation of $e(T)$ which follows from Taylor's formula.
It will be the starting point for our error estimates in Sections \ref{sec:Heat} and \ref{sec:Wave}. For $t\in [0,T]$, $\theta\in[0,1]$ and $y\in U_1$ set
\begin{align*}
F(t)&:=\tilde E(t) B-E(t)B,\\
\Psi_1(t,\theta,y)&:=(1-\theta)\Big\langle u_{xx}\big(t,\tilde Y(t)+ E(T-t)By+\theta F(T-t)y\big)F(T-t)y\,,\,F(T-t)y\Big\rangle_H,\\
\Psi_2(t,\theta,y)&:=\Big\langle u_{xx}\big(t,\tilde Y(t)+\theta E(T-t)By\big)E(T-t)By\,,\,F(T-t)y\Big\rangle_H.
\end{align*}

\begin{corollary}\label{cor:errorRep}
In the setting of Theorem~\ref{thm:errorRep} we have
\begin{equation}\label{eq:errorRep3}
\bE\int_0^T\int_{U_1}\int_0^1\big\{|\Psi_1(t,\theta,y)|+|\Psi_2(t,\theta,y)|\big\}\,\dl\theta\,\nu(\dl y)\,\dl t<\infty,
\end{equation}
and the following alternative error representation holds:
\begin{equation}\label{eq:errorRep4}
\begin{aligned}
e(T)&=\bE\big\{u(0,\tilde E(T)X_0)-u(0,E(T)X_0)\big\}\\
&\quad +\bE\int_0^T\int_{U_1}\int_0^1\big\{\Psi_1(t,\theta,y)+\Psi_2(t,\theta,y)\big\}\,\dl\theta\,\nu(\dl y)\,\dl t.
\end{aligned}
\end{equation}
\end{corollary}

\begin{proof}
The integrand of the iterated integral in \eqref{eq:errorRep2} can be rewritten as
\begin{align*}
u\big(t,\tilde Y(t)+&\tilde E(T-t) By\big)-u\big(t,\tilde Y(t)+E(T-t)By\big)-\big\langle u_x(t,\tilde Y(t)),F(T-t)y\big\rangle_H\\
&=\Big\{u\big(t,\tilde Y(t)+\tilde E(T-t) By\big)-u\big(t,\tilde Y(t)+E(T-t)By\big)\\
&\quad-\big\langle u_x\big(t,\tilde Y(t)+E(T-t)By\big),F(T-t)y\big\rangle_H\Big\}\\
&\quad+\big\langle u_x\big(t,\tilde Y(t)+E(T-t)By\big)-u_x(t,\tilde Y(t)),F(T-t)y\big\rangle_H\\
&=\int_0^1\big\{\Psi_1(t,\theta,y)+\Psi_2(t,\theta,y)\big\}\,\dl\theta,
\end{align*}
where the last step is due to Taylor's formula. By \eqref{eq:errorRep1} we have
\[\bE\int_0^T\int_{U_1}\Big|\int_0^1\big\{\Psi_1(t,\theta,y)+\Psi_2(t,\theta,y)\big\}\,\dl\theta\Big|\,\nu(\dl y)\,\dl t<\infty.\]
The stronger assertion \eqref{eq:errorRep3} follows from the boundedness of $G'':H\to\bo(H)$, Lemma~\ref{lem:integrandIsomorphism}, \eqref{eq:assEB1} and \eqref{eq:assEBtilde1}.
\end{proof}

\subsection{Proof of the error representation formula}

In this subsection, we give the proof of Theorem~\ref{thm:errorRep}.

For $\xi\in L^2(\Omega,\cF_t,\bP;H)$ we have
\begin{equation*}\label{eq:Eu(t,xi)}
\bE\big(G(Z(T,t,\xi))\big)
=\int_H\int_HG(x+y)\,\bP_{\int_t^TE(T-s)B\,\dl L(s)}(\dl y)\bP_\xi(\dl x)
=\bE\big(u(t,\xi)\big)
\end{equation*}
by \eqref{eq:defZttauxi}, \eqref{eq:defuxt}, the independence of $\int_t^TE(T-s)B\,\dl L(s)$ and $\cF_t$, and Fubini's theorem.
Since $X(T)=Y(T)$ and $\tilde X(T)=\tilde Y(T)$ it follows that
\begin{equation}\label{eq:proofErrRep1}
\begin{aligned}
e(T)&=\bE\big(G(\tilde Y(T))-G(Y(T))\big)\\
&=\bE\big(G(Z(T,T,\tilde Y(T)))-G(Z(T,0,Y(0)))\big)\\
&=\bE\big(u(T,\tilde Y(T))-u(0,Y(0))\big)\\
&=\bE\big(u(0,\tilde Y(0))-u(0,Y(0))\big)+\bE\big(u(T,\tilde Y(T))-u(0,\tilde Y(0))\big).
\end{aligned}
\end{equation}
By \eqref{eq:defY} and \eqref{eq:defYtilde}, the first term in the last line equals $\bE(u(0,\tilde E(T)X_0)-u(0,E(T)X_0))$.

To handle the second term in the last line of \eqref{eq:proofErrRep1}, we first assume that $G:H\to\bR$ and $G_x:H\to H$ are bounded, so that $G\in C^2_{\operatorname b}(H,\bR)$. We will remove this restriction later on. We now want to apply It\^{o}'s formula to the function $(t,x)\mapsto u(t,x)$ and the martingale $\tilde Y=(\tilde Y(t))_{t\in[0,T]}$. For this we need the following properties of $u$.

\begin{proposition}\label{prop:backwardKolm}
Let Assumption~\ref{ass:abstractSetting} hold and $G\in C_{\operatorname b}^2(H,\bR)$. The function $u:[0,T]\times H\to\bR,\;(t,x)\mapsto u(t,x)$ defined in \eqref{eq:defuxt} and its Fr\'{e}chet partial derivatives $u_x$, $u_{xx}$ are continuous and bounded on $[0,T]\times H$. The time derivative $u_t$ of $u$ exists on $[0,T)\times H$ and is continuous.
Moreover, for every $\epsilon>0$ there exists some $C_\epsilon<\infty$ such that
\begin{equation}\label{eq:backwardKolm0}
\int_{U_1} \big|u\big(t,x+E(T-t)By\big)- u(t,x)-\big\langle u_x(t,x),E(T-t)By\big\rangle_H\big|\,\nu(\dl y)<C_\epsilon
\end{equation}
for all $t\in[0,T-\epsilon]$, and $u$ satisfies the backward Kolmogorov equation
\begin{equation}\label{eq:backwardKolm1}
\left.
\begin{aligned}
u_t(t,x)&=-\int_{U_1} \Big\{u\big(t,x+E(T-t)By\big)-u(t,x)-\big\langle u_x(t,x),E(T-t)By\big\rangle_H\Big\}\,\nu(\dl y),\\
&\hspace{2cm}(t,x)\in [0,T)\times H,\\
u(T,x)&=G(x),\quad x\in H.
\end{aligned}
\;\right\}
\end{equation}
\end{proposition}

\begin{proof}

We begin with the continuity and boundedness of $u$, $u_x$ and $u_{xx}$. The boundedness is obvious by the definition \eqref{eq:defuxt} of $u$ and by \eqref{eq:u_xu_xx}. Pick $0\leq s\leq t\leq T$, $x,y\in H$. Using \eqref{eq:defuxt}, Jensen's inequality, the mean value theorem, \eqref{eq:defZttauxi} and It\^{o}'s isometry,  we have
\begin{equation*}
\begin{aligned}
|u(t,x)-u(s,y)|^2&\leq\bE\big(|G(Z(T,t,x))-G(Z(T,s,y))|^2\big)\\
&\leq \sup_{x\in H}\nnrm{G'(x)}H^2\,\bE\Big(\gnnrm{x-y-\int_s^t E(T-r)B\,\dl L(r)}H^2\Big)\\
&\leq 2\sup_{x\in H}\nnrm{G'(x)}H^2\Big(\nnrm{x-y}H^2+\int_s^t\nnrm{E(T-r)B}{\hs(U_0,H)}^2\,\dl r\Big).
\end{aligned}
\end{equation*}
Thus, the continuity of $u$ follows from \eqref{eq:assEB1} and the boundedness of $G'$.
Since $u_x(t,x)=\bE\,G'(Z(T,t,x))$, the continuity of $u_x:[0,T]\times H\to H$ follows analogously from the boundedness of $G''$.
To show the continuity of $u_{xx}:[0,T]\times H\to\bo(H)$, define $F\in C_{\operatorname b}(H\times H;\bR)$ by
\[F(x,y):=\nnrm{G''(x)-G''(y)}{\bo(H)},\quad x,\,y\in H,\]
and fix $(t,x)\in[0,T]\times H$, $((t_k,x_k))_{k\in\bN}\subset[0,T]\times H$ with $(t_k,x_k)\to (t,x)$ as $k\to\infty$. Note that $Z(T,t_k,x_k)\to Z(T,t,x)$ in $L^2(\Omega;H)$ by It\^{o}'s isometry. As a consequence, $(Z(T,t,x),Z(T,t_k,x_k))\to (Z(T,t,x),Z(T,t,x))$ in distribution (on $H\times H$) and we obtain
\begin{align*}
\nnrm{u_{xx}(t,x)-u_{xx}(t_k,x_k)}{\bo(H)}
&\leq \bE\, F\big(Z(T,t,x),Z(T,t_k,x_k)\big)\\
&\xrightarrow{k\to\infty}\bE\, F\big(Z(T,t,x),Z(T,t,x)\big)=0,
\end{align*}
yielding the continuity of $u_{xx}$.

By Taylor's formula and Lemma~\ref{lem:integrandIsomorphism},
\begin{equation*}
\begin{aligned}
\int_{U_1} \big|u\big(t,x+E(T-t)By&\big)- u(t,x)-\big\langle u_x(t,x),E(T-t)By\big\rangle_H\big|\,\nu(\dl y)\\
&\leq\;\frac12\sup_{x\in H}\nnrm{G''(x)}{\bo(H)}\int_{U_1}\nnrm{E(T-t)By}{H}^2\nu(\dl y)\\
&=\;\frac12\sup_{x\in H}\nnrm{G''(x)}{\bo(H)}\nnrm{E(T-t)B}{\hs(U_0,H)}^2.
\end{aligned}
\end{equation*}
Using Assumption~\ref{ass:abstractSetting}(v), condition \eqref{eq:backwardKolm0} follows with $$C_\epsilon=1/2\sup_{x\in H}\nnrm{G''(x)}{\bo(H)}\nnrm{\Phi_\epsilon}{\hs(U_0,H)}^2.$$

In order to verify the Kolmogorov equation \eqref{eq:backwardKolm1}, we first note that for fixed $t\in[0,T]$ the $H$-valued random variables $\int_0^tE(s)B\,\dl L(s)$ and $\int_{T-t}^{T} E(T-s)B\,\dl L(s)$ have the same distribution, so that
\begin{equation}\label{eq:backwardKolm2}
v(t,x):=\bE\,G\big(x+\int_0^{t}E(s)B\,\dl L(s)\big)=u(T-t,x),\quad (t,x)\in [0,T]\times H.
\end{equation}
Next, we fix $x\in H$ and apply It\^{o}'s formula \cite[Theorem~3.6]{MRT13} to the function $y\mapsto G(x+y)$ and the martingale $M=(M(t))_{t\in[0,T]}:=(\int_0^t E(s)B\,\dl L(s))_{t\in[0,T]}\in\cM^2_T(H)$. Note that $M$ fits into the setting of \cite{MRT13} since it has the representation
\[M(t)=\int_0^t\int_{U_1} E(s)By\,q(\dl s,\dl y),\quad t\in[0,T],\]
where $q$ is the compensated Poisson random measure on $[0,\infty)\times U_1$ associated to $L$; see the appendix for details.
We obtain
\begin{equation}\label{eq:backwardKolm3}
\begin{aligned}
G(x+M(t))= G(x)&+\int_0^t\int_{U_1}\big\{G\big(x+M(s-)+E(s)By\big)-G\big(x+M(s-)\big)\big\}\,q(\dl s,\dl y)\\
&+\int_0^t\int_{U_1}\big\{G\big(x+M(s)+E(s)By\big)-G\big(x+M(s)\big)\\
&\hspace{4.4cm}-\big\langle G'\big(x+M(s)\big),E(s)By\big\rangle_H\big\}\,\nu(\dl y)\,\dl s,
\end{aligned}
\end{equation}
where the integrand appearing in the integral w.r.t.\ $q$ belongs to $L^2(\Omega_T\times U_1,\bP_T\otimes\nu;\bR)$ as a consequence of Taylor's formula, the boundedness of $G'$, Lemma~\ref{lem:integrandIsomorphism} and \eqref{eq:assEB1}. Similarly, the second integral in \eqref{eq:backwardKolm3} exists for all $\omega\in\Omega$ and belongs to $L^1(\Omega;\bR)$
since
\begin{equation*}\label{eq:backwardKolm4}
\begin{aligned}
\int_0^t\int_{U_1}\big|G\big(x+M(s)+E(s)By\big)-G\big(&x+M(s)\big)-\big\langle G'\big(x+M(s)\big),E(s)By\big\rangle_H\big|\,\nu(\dl y)\,\dl s\\
&\leq\;\frac12\sup_{x\in H}\nnrm{G''(x)}{\bo(H)}\int_0^t\nnrm{E(s)B}{\hs(U_0,H)}^2\dl s.
\end{aligned}
\end{equation*}
Taking expectations in \eqref{eq:backwardKolm3} and using the martingale property of the integral w.r.t.\ $q$ yields
\begin{equation}\label{eq:backwardKolm5}
v(t,x)=G(x)+\int_0^t\int_{U_1}\big\{v(s,x+E(s)By)-v(s,x)-\langle v_x(s,x),E(s)By\rangle_H\big\}\,\nu(\dl y)\,\dl s.
\end{equation}
By the fundamental theorem of calculus, \eqref{eq:backwardKolm1} follows from \eqref{eq:backwardKolm2}, \eqref{eq:backwardKolm5} if the mapping
\begin{equation}\label{eq:backwardKolm7}
(0,T]\ni s\to \int_{U_1}\big\{v(s,x+E(s)By)-v(s,x)-\langle v_x(s,x),E(s)By\rangle_H\big\}\,\nu(\dl y)\in\bR.
\end{equation}
is continuous.

Note that we cannot apply directly the continuity theorem for parameter-dependent integrals to show the continuity of the mapping \eqref{eq:backwardKolm7}. The reason is that the term $E(s)By$ in the integral in \eqref{eq:backwardKolm7} is defined only in an $L^2([0,T]\times U_1,\lambda\otimes\nu;H)$-sense, cf.~Lemma~\ref{lem:integrandIsomorphism} and Remark~\ref{not:Phix}, so that we have no information about the continuity of $(0,T]\in s\mapsto E(s)By\in H$ for fixed $y\in U_1$. Therefore, we use an approximation argument: For $s\in(0,T]$, $x\in H$, $y\in U_1$ and $k\in\bN$ set
\begin{align*}
f(s,x,y)&:=v(s,x+E(s)By)-v(s,x)-\langle v_x(s,x),E(s)By\rangle_H,\\
f_k(s,x,y)&:=f(s,x,\Pi_ky),
\end{align*}
where $\Pi_k$ is the orthogonal projection of $U_1$ onto $\text{span}\{f_1,\ldots,f_k\}$, $(f_k)_{k\in\bN}\subset U_0$ being an orthonormal basis of $U_1$ as in Lemma~\ref{lem:integrandIsomorphism}. For fixed $x\in H$,  $f(s,x,y)$ is defined in an $L^2([0,T]\times U_1,\lambda\otimes\nu;\bR)$-sense whereas $f_k(s,x,y)$ is defined pointwise. The continuity theorem for parameter-dependent integrals and the strong continuity of $(E(t))_{t\geq 0}$ yield the continuity of in $\int_{U_1}f_k(s,x,y)\nu(\dl y)$ in $(s,x)\in[0,T]\times H$. Moreover, we have $f_k(s,x,\cdot)\stackrel{k\to\infty}{\longrightarrow} f(s,x,\cdot)$ in $L^1(U_1,\nu;\bR)$, uniformly in $(s,x)\in[\epsilon,T]\times H$ for all $\epsilon>0$. Indeed, setting $\Pi^ky:=y-\Pi_ky$ and using Taylor's theorem, Lemma~\ref{lem:integrandIsomorphism} and Assumption~\ref{ass:abstractSetting}(v), we obtain
\begin{align*}
&\int_{U_1}|f(s,x,y)-f_k(s,x,y)|\,\nu(\dl y)\\
&\leq\int_{U_1}\int_0^1\big|\big\langle v_{xx}\big(s,x+E(s)B(\Pi_ky+\theta\Pi^ky)\big) E(s)B\Pi^ky,E(s)B\Pi^ky\big\rangle_H\big|(1-\theta)\,\dl\theta\,\nu(\dl y)\\
&\quad+\int_{U_1}\int_0^1\big|\big\langle v_{xx}\big(s,x+\theta E(s)B\Pi_ky\big)E(s)B\Pi_ky,E(s)B\Pi^ky\big\rangle_H\big|\,\dl\theta\,\nu(\dl y)\\
&\leq\sup_{x\in H}\nnrm{G''(x)}{\bo(H)}\big(\nnrm{E(s)B\Pi^k}{\hs(U_0,H)}^2+\nnrm{E(s)B\Pi_k}{\hs(U_0,H)}\nnrm{E(s)B\Pi^k}{\hs(U_0,H)}\big)\\
&\leq\sup_{x\in H}\nnrm{G''(x)}{\bo(H)}\big(\nnrm{\Phi_\epsilon\Pi^k}{\hs(U_0,H)}^2+\nnrm{\Phi_\epsilon\Pi_k}{\hs(U_0,H)}\nnrm{\Phi_\epsilon\Pi^k}{\hs(U_0,H)}\big)
\end{align*}
for all $s\in[\epsilon,T]$ and some $\Phi_\epsilon\in\hs(U_0,H)$. The expression in the last line tends to zero as $k\to\infty$.
As a consequence, $\int_{U_1}f_k(s,x,y)\nu(\dl y)\xrightarrow{k\to\infty}\int_{U_1}f(s,x,y)\nu(\dl y)$ uniformly in $(s,x)\in[\epsilon,T]\times H$. Thus, the continuity of $\int_{U_1}f_k(s,x,y)\nu(\dl y)$ in $(s,x)\in[0,T]\times H$ implies the continuity of $\int_{U_1}f(s,x,y)\nu(\dl y)$ in $(s,x)\in(0,T]\times H$. In particular, we obtain the continuity of the mapping \eqref{eq:backwardKolm7} as well as the continuity of $u_t$ on $[0,T)\times H$.
\end{proof}

For $G\in C_{\operatorname b}^2(H,\bR)$, the regularity assertions in Propostition~\ref{prop:backwardKolm} allow us to apply It\^{o}'s formula \cite[Theorem~3.6]{MRT13} to the function $(t,x)\mapsto u(t,x)$ and the $H$-valued martingale $\tilde Y=(\tilde Y(t))_{t\in[0,T]}$ defined in \eqref{eq:defYtilde}. Note that $\tilde Y$ fits into the setting of \cite{MRT13} since it has the representation
\begin{equation}\label{eq:proofErrRep1.5}
\tilde Y(t)=\tilde E(T)X_0+\int_0^t\int_{U_1}\tilde E(T-s) B y\,q(\dl s,\dl y),\quad t\in[0,T],
\end{equation}
where again $q$ is the compensated Poisson random measure on $[0,\infty)\times U_1$ associated with $L$ as described in the appendix. Equality \eqref{eq:proofErrRep1.5} is a consequence of \eqref{eq:assEBtilde1}, Lemma~\ref{lem:integrandIsomorphism}, Remark~\ref{not:Phix} and
Lemma~\ref{lem:comparisonIntegrals}. For $T'\in(0,T)$ we obtain
\begin{equation}\label{eq:proofErrRep2}
\begin{aligned}
u(T',\tilde Y(T'))\,&=\,u(0,\tilde Y(0))+\int_0^{T'}u_t(t,\tilde Y(t))\,\dl t\\
&\quad+\int_0^{T'}\int_{U_1}\big\{u\big(t,\tilde Y(t-)+\tilde E(T-t) By\big)-u(t,\tilde Y(t-))\big\}\,q(\dl s,\dl y)\\
&\quad+\int_0^{T'}\int_{U_1}\big\{u\big(t,\tilde Y(t)+\tilde E(T-t) By\big)-u(t,\tilde Y(t))\\
&\quad-\big\langle u_x(t,\tilde Y(t)),\tilde E(T-t) By\big\rangle_H\big\}\,\nu(\dl y)\,\dl s.
\end{aligned}
\end{equation}
Using the boundedness of $u$, $u_x$ and $u_{xx}$, \eqref{eq:backwardKolm1}, \eqref{eq:assEB1} and applying similar arguments as in the proof of Proposition~\ref{prop:backwardKolm}, one sees that all terms in \eqref{eq:proofErrRep2} are well-defined and integrable w.r.t.\ $\bP$. Thus, we can take expectations and use the martingale property of the integral w.r.t.\ $q$ and the backward Kolmogorov equation \eqref{eq:backwardKolm1} to obtain
\begin{equation}\label{eq:proofErrRep3}
\begin{aligned}
\bE\big(&u(T',\tilde Y(T'))-u(0,\tilde Y(0))\big)=\\
&\bE\int_0^{T'}\int_{U_1}\Big\{u\big(t,\tilde Y(t)+\tilde E(T-t) By\big)-u\big(t,\tilde Y(t)+E(T-t)By\big)\\
&\qquad\qquad\qquad\qquad-\big\langle u_x(t,\tilde Y(t)),\big(\tilde E(T-t) B-E(T-t)B\big)y\big\rangle_H\Big\}\,\nu(\dl y)\,\dl t
\end{aligned}
\end{equation}
for all $T'\in(0,T)$. Taking the limit $T'\to T$ on both sides of \eqref{eq:proofErrRep3}, we can replace $T'$ by $T$. Here we used the stochastic continuity of $\tilde Y$ and the continuity of $u$ for the limit on the left hand side. For the limit on the right hand side we used \eqref{eq:errorRep1}, which is again a consequence of Taylor's formula, the boundedness of $G''$, \eqref{eq:assEB1}, \eqref{eq:assEBtilde1} and Lemma~\ref{lem:integrandIsomorphism}, using similar arguments as in the proof of Proposition~\ref{prop:backwardKolm}. The combination of \eqref{eq:proofErrRep1} and \eqref{eq:proofErrRep3} yields the error representation formula \eqref{eq:errorRep2} for the case $G\in C_{\operatorname b}^2(H,\bR)$.

Finally, we consider the general case of a test function $G\in C^2(H,\bR)$ such that $\sup_{x\in H}\nnrm{G''(x)}{\bo(H)}<\infty$.
For $\epsilon>0$, define $G_\epsilon\in C_{\operatorname b}^2(H,\bR)$ by
\begin{equation*}
G_\epsilon(x):=e^{-\epsilon\nnrm{x}H^2} G(x),\quad x\in H.
\end{equation*}
The Fr\'{e}chet derivatives $G_\epsilon'(x)\in H$ and $G_\epsilon''(x)\in\bo(H)$ are given by
\begin{align*}
G_\epsilon'(x)&=e^{-\epsilon\nnrm{x}H^2}\big(G'(x)-2\epsilon G(x)x\big),\\
G_{\epsilon}''(x)&=e^{-\epsilon\nnrm{x}H^2}\big[G''(x)+2\epsilon\big(\langle G'(x),\arg\rangle_H x-\langle x,\arg\rangle_H G'(x)\big)-4\epsilon^2 G(x)\langle x,\arg\rangle_Hx\big],
\end{align*}
so that the boundedness of $G_\epsilon:H\to\bR$, $G_\epsilon':H\to H$ and $G_\epsilon'':H\to\bo(H)$ follows from \eqref{eq:growthGG'}. Moreover, note that $G_\epsilon(x)\xrightarrow{\epsilon\to0}G(x)$ in $\bR$ and $G_\epsilon'(x)\xrightarrow{\epsilon\to0}G'(x)$ in $H$ for all $x\in H$, as well as
$
\sup_{\epsilon\in(0,1]}\sup_{x\in H}\nnrm{G_\epsilon''(x)}{\bo(H)}<\infty.
$
The latter is a consequence of \eqref{eq:growthGG'} and the standard estimate $\sup_{s>0}s^ne^{-s}<n!$, $n\in\bN$.  We set
\begin{equation}\label{eq:defuxtEps}
u_\epsilon(t,x):=\bE\,G_\epsilon(Z(T,t,x)),\quad (t,x)\in [0,T]\times H.
\end{equation}
By what has been shown above, the assertion of Theorem~\ref{thm:errorRep} holds with $u$ replaced by $u_\epsilon$. By a dominated convergence argument using \eqref{eq:growthGG'}, we obtain that $u_\epsilon(t,x)\xrightarrow{\epsilon\to0} u(t,x)$ in $\bR$ and $(u_\epsilon)_x(t,x)\xrightarrow{\epsilon\to0} u_x(t,x)$ in $H$ for all $(t,x)\in [0,T]\times H$. For all $t\in[0,T]$ and $y\in U_1$, we have
\begin{equation}\label{eq:proofErrRep4}
\begin{aligned}
\Big|u_\epsilon\big(t,&\tilde Y(t)+\tilde E(T-t) By\big)-u_\epsilon\big(t,\tilde Y(t)+E(T-t)By\big)\\
&\qquad-\big\langle (u_\epsilon)_x(t,\tilde Y(t)),\big(\tilde E(T-t) B-E(T-t)B\big)y\big\rangle_H\Big|\\
&\leq \sup_{\epsilon\in(0,1]}\sup_{x\in H}\nnrm{G_\epsilon''(x)}{\bo(H)}\Big(\nnrm{\tilde E(T-t) By}H^2+\nnrm{E(T-t)B}H^2\big)
\end{aligned}
\end{equation}
due to Taylor's theorem and the fact that $(u_\epsilon)_{xx}(t,x)=\bE G_\epsilon''(Z(T,t,x))$. Note that the term right hand side of \eqref{eq:proofErrRep4} is integrable w.r.t. $\dl t\,\nu(\dl y)$ as a consequence of \eqref{eq:assEB1}, \eqref{eq:assEBtilde1} and Lemma~\ref{lem:integrandIsomorphism}. Thus, considering \eqref{eq:errorRep1} and \eqref{eq:errorRep2} with $u$ replaced by $u_\epsilon$, we can use dominated convergence as $\epsilon\to0$  to finish the proof.

\section{Applications to parabolic equations}
\label{sec:Heat}

\subsection{The stochastic heat equation}\label{subsec:she} Here we give a detailed error analysis of a space-time discretization of the linear stochastic heat equation with additive L\'{e}vy noise.

Let $\dom\subset\bR^d$ be a bounded convex domain. 
Let $\Lambda:=-\Delta=-\sum_{j=1}^d\partial^2/\partial\xi_j^2$ be the Laplace operator on $L^2(\dom)$ with zero-Dirichlet boundary condition, i.e., with domain $D(\Lambda):=\{v\in H^1_0{(\dom)}:~\Lambda u\in L^2(\dom)\}$, where $\Lambda u$ is understood in the distributional sense, see \cite[Example 3.4.7]{ABHN11}. As usual, $H^n(\dom)$ denotes the $L^2$-Sobolev space of order $n\in\bN_0$ on $\dom$  and $H^1_0(\dom)$ is the $H^1(\dom)$-closure of the space $C_c^\infty(\dom)$ of compactly supported test functions.
Then, setting
\[H:=U:=L^2(\dom),\quad (A,D(A)):=(\Lambda,D(\Lambda)),\quad B:=\id{L^2(\dom)},\]
the abstract equation \eqref{eq:mainEq} becomes the stochastic heat equation \eqref{eq:SHE}. It is not difficult to see that the condition $\nnrm{\Lambda^{-1/2}Q^{1/2}}{\hs(H)}<\infty$ implies \eqref{eq:assEB1}, where
\begin{equation}\label{eq:E(t)SHE}
(E(t))_{t\geq0}:=(e^{-t\Lambda})_{t\geq0}\subset\bo(H)
\end{equation}
is the semigroup generated by $-A=-\Lambda$. Hence, there exists a unique weak solution $X=(X(t))_{t\geq0}$ to Eq.~\eqref{eq:SHE}, given by the variation-of-constants formula \eqref{eq:X}. Furthermore, for some $C>0$ independent of $T$,
\begin{equation}\label{eq:hstab}
\nnrm{X(T)}{L^2(\Omega,H)}\le C (\nnrm{\Lambda^{-1/2}Q^{1/2}}{\hs( H)}+\nnrm{X_0}{L^2(\Omega,H)}).
\end{equation}
In the sequel, we use the smoothness spaces $\dot H^\alpha$, $\alpha\in\bR$, defined by
\begin{align*}
\dot H^\alpha&:=D(\Lambda^{\alpha/2})\\
&:=\Big\{v=\sum_{k=1}^\infty v_k\varphi_k:(v_k)_{k\in\bN}\subset\bR,\;|v|_{\alpha}:=\nnrm{\Lambda^{\alpha/2}v}{L^2(\dom)}=\Big(\sum_{k=1}^\infty \lambda_k^\alpha v_k^2\Big)^{1/2}<\infty\Big\},
\end{align*}
where $(\varphi_k)_{k\in\bN}\subset D(\Lambda)$ is an orthonormal basis of $L^2(\dom)$ consisting of eigenfunctions of $\Lambda $ and $(\lambda_k)_{k\in\bN}\subset(0,\infty)$ is the corresponding sequence of eigenvalues; compare \cite[Chapters 3 and 19]{Tho06}. They are Hilbert spaces and one has the identities $\dot H^0=H=L^2(\dom)$, $\dot H^1=H^1_0(\dom)$ and $\dot H^2=D(\Lambda)=H^2(\dom)\cap H^1_0(\dom)$, where the natural norms of the respective spaces are equivalent. The latter equality is a consequence of the elliptic regularity estimate
\begin{equation}\label{eq:ellreg}
\|v\|_{H^2}\le C\|v\|_{\dot H^2},\quad v\in \dot{H}^2,
\end{equation}
for bounded convex domains, see \cite[Corollary 1]{Fromm93}. Without the convexity assumption one can still define the $\dot H^\alpha$ spaces as above but one does not obtain a characterization of $D(\Lambda)=\dot{H}^2$ in terms of classical Sobolev spaces.
For negative $\alpha$, the elements of $\dot H^\alpha$ are formal sums and we identify them with elements of $L^2(\dom)$ if $\sum_{k=1}^\infty v_k^2<\infty$, so that $\dot H^\alpha$ is the closure of $L^2(\dom)$ w.r.t.\ the $|\cdot|_\alpha$-norm.
\begin{remark}\label{rem:interpolation}
The spaces $\dot H^\alpha$, $\alpha\in\bR$, can be obtained by both real and complex interpolation: For $\alpha=(1-\theta)\alpha_0+\theta\alpha_2$, $\theta\in(0,1)$, one has $\dot H^\alpha=(\dot H^{\alpha_0},\dot H^{\alpha_1})_{\theta,2}=[\dot H^{\alpha_0},\dot H^{\alpha_1}]_\theta$ with equivalent norms, where $(\cdot,\cdot)_{\theta,2}$ and $[\cdot,\cdot]_\theta$ denotes real interpolation with summability parameter $q=2$ and complex interpolation, respectively. This follows, e.g., from \cite[Theorem 1.18.5]{Tri78} and the fact that the spaces $\dot H^\alpha$, $\alpha\in\bR$, are isometrically isomorphic to weighted $\ell^2$-spaces. We will frequently use the corresponding interpolation inequalities in this and the next section.
\end{remark}

For the spatial discretization of Eq.~\eqref{eq:SHE}, we consider a family of finite-dimensional spaces $(S_h)_{h>0}\subset H^1_0(\dom)$.
 Unless otherwise stated, we endow the spaces $S_h$ with the inner product $\langle\cdot,\cdot\rangle_H$ and the norm $\nnrm{\cdot}H$.  By $P_h:H\to S_h$ and $\Pi_h:\dot H^1\to S_h$ we denote the orthogonal projections with respect to the inner products in $H$ and $\dot H^1$, respectively. The discrete Laplacian $\Lambda_h:S_h\to S_h$ is defined by
\begin{equation}\label{eq:discreteLaplace}
\langle\Lambda_h v,w\rangle_{L^2(\dom)}=\langle\nabla v,\nabla w\rangle_{L^2(\dom;\bR^2)},\quad v,w\in S_h.
\end{equation}
Our assumption on the spatial approximation is formulated via the following estimate on the Ritz projection $\Pi_h$,
\begin{equation}\label{eq:basicFEMestimate1}
\nnrm{\Pi_hv-v}{L^2(\dom)}\leq C h^\beta|v|_\beta,\quad v\in\dot H^\beta,\;1\leq\beta\leq2.
\end{equation}
This holds for example, if $S_h$ is consisting of piecewise linear functions with respect to a family of triangulations of $\dom$. The parameter $h$ corresponds to the maximal mesh size of the triangulation, see, e.g., \cite[Lemma~1.1]{Tho06} or \cite[Section~5.4]{LarTho03}.

The time discretization of Eq.~\eqref{eq:SHE} on a finite interval $[0,T]$ is done via the implicit Euler scheme with time step $\dt=T/N$, $N\in\bN$, and grid points $t_n=n\dt$, $n=0,\ldots N$. For $h>0$ and $N\in\bN$, the discretization $(X^n_{h,\dt})_{n=1,\ldots,N}$ of $(X(t))_{t\in[0,T]}$ in space and time is given as the solution to
\begin{equation}\label{eq:schemeSHE}
X^n_{h,\dt}-X^{n-1}_{h,\dt}+\dt\Lambda_hX^n_{h,\dt}=P_h(L(t_n)-L(t_{n-1})),\quad n=1,\ldots,N;\quad X^0_{h,\dt}=P_hX_0.
\end{equation}

\begin{remark}[strong error] \label{rem:strongErrorSHE}
If the covariance operator $Q\in\bo(H)$ of $L$ is such that
\begin{equation}\label{eq:SHEassQ}
\nnrm{\Lambda^{\frac{\beta-1}2}Q^{\frac12}}{\hs(H)}<\infty
\end{equation}
for some $\beta\geq0$, then the solution $X(t)$ takes values in $\dot H^\beta$ for all $t>0$. For the Gaussian case, i.e., the case where $L$ in \eqref{eq:SHE} is a $Q$-Wiener process, it has been shown in \cite[Theorem~1.2]{Yan05} that, if \eqref{eq:SHEassQ} holds and $X_0\in L^2(\Omega,\cF_0,\bP;\dot H^\beta)$ for some $\beta\in(0,1]$, then the scheme \eqref{eq:schemeSHE} has strong convergence of order $\beta$ in space and $\beta/2$ in time:
\[\nnrm{X^n_{h,\dt}-X(t_n)}{L^2(\Omega;H)}\leq C(h^\beta+(\dt)^{\frac\beta2}),\quad n=0,\ldots,N.\]
Unlike weak error estimates, strong $L^2$-error estimates are the same in the Gaussian case and in our setting, since the only stochastic tool that is needed is It\^{o}'s isometry~\eqref{eq:ItoIsom} which looks the same
if the driving noise is a L\'{e}vy process which is an $L^2$-martingale.
Thus the strong error result in \cite[Theorem~1.2]{Yan05} carries over one-to-one to our setting.
\end{remark}

\begin{remark}
The $S_h$-valued random variables $P_h(L(t_n)-L(t_{n-1}))$ in \eqref{eq:schemeSHE} can be defined in two ways. On the one hand, we may set
\[P_h(L(t_n)-L(t_{n-1})):=L^2(\Omega;S_h)\text{-}\lim_{K\to\infty}\sum_{k=1}^K(L_k(t_n)-L_{k}(t_{n-1}))P_he_k,\]
with an orthonormal basis $(e_k)_{k\in\bN}$ of $U_0$ and real-valued uncorrelated L\'{e}vy processes $L_k=(L_k(t))_{t\geq0}$, $k\in\bN$, as in Remark~\ref{rem:approxL}. The limit exists since, by the finite-dimensionality of $S_h$, one has $P_h\in\hs(H,S_h)=\hs(U,S_h)\subset\hs(U_0,S_h)$.
On the other hand, we can extend the orthogonal projection $P_h:H\to S_h$ to a generalized $L^2$-projection $P_h:\dot H^{-1}\to S_h$ defined by
\[\langle P_hv,w\rangle_{H}=\langle v,w\rangle_{\dot H^{-1}\times\dot H^1},\quad v\in \dot H^{-1},\;w\in S_h.\] Then, the assumption $\nnrm{\Lambda^{-1/2}Q^{1/2}}{\hs(H)}<\infty$ implies that we can take $U_1:=D(\Lambda^{-1/2})=\dot H^{-1}$ as the state space of $L$, so that the expression $P_h(L(t_n)-L(t_{n-1}))$ makes sense $\omega$-wise. Obviously, both definitions are compatible.
In practice, one has to find a suitable way to sample (an approximation of) the discretized noise increment $P_h(L(t_n)-L(t_{n-1}))$. We do not treat this problem in the present paper but refer to \cite{BarLan12a, DunHauPro12} and \cite[Remark~4]{LinSch13} for related considerations.
\end{remark}

  With $R(\lambda):=1/(1+\lambda)$ and $E_{h,\dt}:=R(\dt\Lambda_h):=(I+\dt\Lambda_h)^{-1}$ as well as $E_{h,\dt}^n:=R^n(\dt\Lambda_h):=((I+\dt\Lambda_h)^{-1})^n$, the scheme \eqref{eq:schemeSHE} can be rewritten as
\begin{equation}\label{eq:schemeSHE2}
X^n_{h,\dt}=E_{h,\dt}^nP_hX_0+\sum_{j=1}^n E^{n-j+1}_{h,\dt}P_h(L(t_j)-L(t_{j-1})),\quad n=0,\ldots,N.
\end{equation}
For $t\in[0,T]$, let $\tilde E(t)=\tilde E_{h,\dt}(t)\in\bo(H)$ be defined by
\begin{equation}\label{eq:EtildeSHE}
\tilde E(t)=\tilde E_{h,\dt}(t):=\one_{\{0\}}(t)P_h+\sum_{n=1}^N\one_{(t_{n-1},t_n]}(t)E_{h,\dt}^nP_h
\end{equation}
and set
\begin{equation}\label{eq:XtildeSHE}
\tilde X(t)=\tilde X_{h,\dt}(t):=\tilde E_{h,\dt}(t)X_0+\int_0^t\tilde E_{h,\dt}(t-s)\,\dl L(s).
\end{equation}
Then $X^n_{h,\dt}=\tilde X_{h,\dt}(t_n)$ $\bP$-almost surely. This follows from the construction of the stochastic integral, using an approximation argument and It\^{o}'s isometry \eqref{eq:ItoIsom}.

The following deterministic estimates will be used in the proof of our weak error result stated in Theorem~\ref{thm:SHE} below.
\begin{lemma}
The operators $E(t)$ and $\tilde E(t)=\tilde E_{h,\dt}(t)$ defined in \eqref{eq:E(t)SHE} and \eqref{eq:EtildeSHE} satisfy the error estimates
\begin{align}
\nnrm{\tilde E(s)-E(s)}{\bo(H)}&\leq C(h^{2}+\dt)s^{-1},\label{eq:detEstSHE1}\\
\nnrm{\Lambda^\alpha E(s)}{\bo(H)}+\nnrm{\Lambda^\alpha \tilde E(s)}{\bo(H)}&\leq C s^{-\alpha},\quad 0\leq\alpha\leq1/2,\label{eq:detEstSHE2}
\end{align}
$s\in(0,T]$, where $C>0$ does not depend on $h$, $\dt$ and $s$.
\end{lemma}

\begin{proof}
Estimate \eqref{eq:detEstSHE1} follows from
\begin{equation}\label{eq:hest}
\nnrm{E_{h,\dt}^nP_h-E(t_n)}{\bo(H)}\leq C(h^2+\dt)t_n^{-1},
\end{equation}
see, for example, \cite[Theorem~7.7]{Tho06}. We note here that while the latter result is proved under the assumption that $\dom$ has smooth boundary, the proof relies on the availability of \eqref{eq:basicFEMestimate1}, which is our basic assumption, and hence \eqref{eq:hest} is valid in our setting.
For $s\in(t_{n-1},t_n]$ we have
\begin{align*}
\nnrm{(E(t_n)-E(s))v}H&=\nnrm{\Lambda E(s)(E(t_n-s)-\id{H})\Lambda^{-1}v}H\\
&\leq\nnrm{\Lambda E(s)}{\bo(H)}\nnrm{(E(t_n-s)-\id{H})\Lambda^{-1}v}H\\
&\leq Cs^{-1}\dt\nnrm{v}H,
\end{align*}
where we used Theorem~6.13(c),(d) on analytic semigroups in \cite[Chapter 2]{Paz83}.
Estimate~\eqref{eq:detEstSHE2} is due to Theorem~6.13(c) in \cite[Chapter 2]{Paz83}, Lemma~7.3 in  \cite{Tho06}, interpolation, and the fact that $\nrm{A^\alpha v_h}\leq\nrm{A^\alpha_hv_h}$ for $v_h\in S_h$, $0\leq\alpha\leq 1/2$. The latter follows from the basic identity $\nrm{A^{1/2}v_h}=\nrm{A_h^{1/2}v_h}$ and interpolation.
\end{proof}

Here is our result for the weak error of the discretization of the stochastic heat equation.
\begin{theorem}\label{thm:SHE}
Assume that $X_0\in L^2(\Omega,\cF_0,\bP;H)$ and $\nnrm{\Lambda^{(\beta-1)/2}Q^{1/2}}{\hs(H)}<\infty$ for some $\beta\in(0,1]$. Let $(X(t))_{t\geq0}$ be the weak solution \eqref{eq:X} to Eq.~\eqref{eq:mainEq} and let $(X_{h,\dt}^n)_{n=0,\ldots,N}$ be defined by the scheme \eqref{eq:schemeSHE}. Given $g\in C^2(H,\bR)$ with $\sup_{x\in H}\nnrm{g''(x)}{\bo(H)}<\infty$, there exists a constant $C=C(g,T)>0$
that does not depend on $h$ and $\dt$, such that
\[\big|\bE\big(g(X^N_{h,\dt})-g(X(T))\big)\big|\leq C(h^{2\beta}+(\dt)^\beta)|\log(h^2+\dt)|\]
for $h^{2}+ \dt\leq 1/e$.
\end{theorem}

\begin{proof}
We are in the setting of Section~\ref{sec:Setting and preliminaries} with $H=U=L^2(\dom)$, $B=\id H$, and $(E(t))_{t\geq0}$, $(\tilde E(t))_{t\in[0,T]}=(\tilde E_{h,\dt}(t))_{t\in[0,T]}$, $(\tilde X(t))_{t\in[0,T]}=(\tilde X_{h,\dt}(t))_{t\in[0,T]}$ being given by \eqref{eq:E(t)SHE}, \eqref{eq:EtildeSHE}, \eqref{eq:XtildeSHE} respectively. In particular, Assumption~\ref{ass:abstractSetting} is fulfilled. Since $X_{h,\dt}^N=\tilde X(T)$, we can use Corollary~\ref{cor:errorRep} with $G:=g$ to estimate the weak error.
Let $F(t):=\tilde E(t)-E(t)$ be the deterministic error operator.

We begin with the first term on the right hand side of \eqref{eq:errorRep4} in Corollary~\ref{cor:errorRep}.
The stability estimate \eqref{eq:hstab} and the deterministic estimate \eqref{eq:detEstSHE1} yield, for $\max(h^{2},\dt)\leq 1$,
\begin{equation}\label{eq:proofSHE1}
\begin{aligned}
&\big|\bE\big\{u(0,\tilde E(T)X_0)-u(0,E(T)X_0)\big\}\big|\\
&= \big|\bE\big\{u(0,\tilde Y(0))-u(0,Y(0))\big\}\big|\\
&= \Big|\bE\int_0^1\big\langle u_x\big(0,Y(0)+\theta(\tilde Y(0)-Y(0))\big),\tilde Y(0)-Y(0)\big\rangle_{ H}\dl \theta\Big|\\
&= \Big|\bE\int_0^1\big\langle\bE\big(g'(Z(T,0,x))\big)\big|_{x=Y(0)+\theta(\tilde Y(0)-Y(0))},\tilde Y(0)-Y(0)\big\rangle_{H}\dl\theta\Big|\\
&\leq \int_0^1\big\|g'\big(Z\big(T,0,Y(0)+\theta(\tilde Y(0)-Y(0))\big)\big)\big\|_{L^2(\Omega, H)}\dl\theta\, \|(\tilde E(T)-E(T))X_0\|_{L^2(\Omega, H)}\\
&\leq C\big(1+\int_0^1\big\|Z\big(T,0,Y(0)+\theta(\tilde Y(0)-Y(0))\big)\big\|_{L^2(\Omega, H)}\dl\theta\big)\,(h^2+\dt)\,T^{-1}\,\nnrm{X_0}{L^2(\Omega; H)}\\
&\leq C\big(1+\nnrm{\Lambda^{-1/2}Q^{1/2}}{\hs(\dot H^0)}+\nnrm{X_0}{L^2(\Omega; H)}\big)\,\nnrm{X_0}{L^2(\Omega; H)}\,T^{-1}\,(h^{2\beta}+(\dt)^\beta).
%
\end{aligned}
\end{equation}

Next, consider the second term on the right hand side of \eqref{eq:errorRep4}. We estimate the integrals of the functions $\Psi_1$ and $\Psi_2$ separately. Using Lemma~\ref{lem:integrandIsomorphism} and Remark~\ref{not:Phix}, we obtain
\begin{equation}\label{eq:proofSHE2}
\begin{aligned}
\big|\bE\int_0^T&\int_{U_1}\int_0^1\Psi_1(t,\theta,y)\,\dl\theta\,\nu(\dl y)\,\dl t\big|\\
&\leq \sup_{x\in H}\nnrm{g''(x)}{\bo(H)}\int_0^T\int_{U_1}\nnrm{F(T-t)y}H^2\,\nu(\dl y)\,\dl t\\
&=\sup_{x\in H}\nnrm{g''(x)}{\bo(H)}\int_0^T\nnrm{F(T-t)}{\hs(U_0,H)}^2\,\dl t\\
&\leq C \sup_{x\in H}\nnrm{g''(x)}{\bo(H)}(h^{2\beta}+(\dt)^\beta).
\end{aligned}
\end{equation}
The last step is due to the fact that, by It\^{o}'s isometry \eqref{eq:ItoIsom}, the integral in the penultimate line is the square of the strong error $\nnrm{X^N_{h,\dt}-X(T)}{L^2(\Omega;H)}$ for zero initial condition $X_0=0$, which can be estimated as in the Gaussian case \cite[Theorem~1.2]{Yan05}, compare Remark~\ref{rem:strongErrorSHE}. Further, by the Cauchy-Schwarz inequality, Lemma~\ref{lem:integrandIsomorphism}, and the fact that $U_0=Q^{1/2}(U)$,
\begin{equation}\label{eq:proofSHE3}
\begin{aligned}
\big|\bE&\int_0^T\int_{U_1}\int_0^1\Psi_2(t,\theta,y)\,\dl\theta\,\nu(\dl y)\,\dl t\big|\\
&\leq \sup_{x\in H}\nnrm{g''(x)}{\bo(H)}\int_0^T\int_{U_1}\nnrm{E(T-t)y}H\nnrm{F(T-t)y}H\,\nu(\dl y)\,\dl t\\
&\leq\sup_{x\in H}\nnrm{g''(x)}{\bo(H)}\int_0^T\nnrm{E(T-t)}{\hs(U_0,H)}\nnrm{F(T-t)}{\hs(U_0,H)}\,\dl t\\
&\leq\sup_{x\in H}\nnrm{g''(x)}{\bo(H)}\nnrm{\Lambda^{\frac{\beta-1}2}Q^{1/2}}{\hs(H)}^2\int_0^T\nnrm{E(t)\Lambda^{\frac{1-\beta}2}}{\bo(H)}\nnrm{F(t)\Lambda^{\frac{1-\beta}2}}{\bo(H)}\,\dl t
\end{aligned}
\end{equation}
By \eqref{eq:detEstSHE2} we have
\begin{equation}\label{eq:detEstSHE3}
\nnrm{E(t)\Lambda^{\frac{1-\beta}2}}{\bo(H)}=\nnrm{\Lambda^{\frac{1-\beta}2}E(t)}{\bo(H)}\leq Ct^{-\frac{1-\beta}2}
\end{equation}
and
\begin{equation}\label{eq:detEstSHE4}
\nnrm{\Lambda^\alpha F(t)}{\bo(H)}\leq Ct^{-\alpha},\quad 0\leq\alpha\leq1/2.
\end{equation}
Interpolation between \eqref{eq:detEstSHE1} and \eqref{eq:detEstSHE4} with $\alpha=1/2$ gives
\begin{equation}\label{eq:detEstSHE5}
\nnrm{\Lambda^{\frac{1-\beta}2}F(t)}{\bo(H)}\leq C\nnrm{F(t)}{\bo(H)}^\beta\nnrm{\Lambda^{\frac12}F(t)}{\bo(H)}^{1-\beta}\leq C(h^2+\dt)^\beta t^{-\frac{1+\beta}2}.
\end{equation}
Note that $\nnrm{F(t)\Lambda^\alpha}{\bo(H)}=\nnrm{\Lambda^\alpha F(t)}{\bo(H)}$ due to the self adjointness of $\tilde E(t)$, $E(t)$ and $\Lambda^\alpha$. Altogether, using \eqref{eq:detEstSHE3}, \eqref{eq:detEstSHE4} and \eqref{eq:detEstSHE5}, the integral in the last line of \eqref{eq:proofSHE3} can be estimated by
\begin{equation}\label{eq:proofSHE4}
\begin{aligned}
\int_0^T&\nnrm{E(t)\Lambda^{\frac{1-\beta}2}}{\bo(H)}\nnrm{F(t)\Lambda^{\frac{1-\beta}2}}{\bo(H)}\,\dl t\\
&=\Big(\int_0^{h^2+\dt}+\int_{h^2+\dt}^T\Big)\nnrm{\Lambda^{\frac{1-\beta}2}E(t)}{\bo(H)}\nnrm{\Lambda^{\frac{1-\beta}2}F(t)}{\bo(H)}\,\dl t\\
&\leq C\int_0^{h^2+\dt}t^{-\frac{1-\beta}2}t^{-\frac{1-\beta}2}\,\dl t+C\int_{h^2+\dt}^Tt^{-\frac{1-\beta}2}(h^2+\dt)^\beta t^{-\frac{1+\beta}2}\,\dl t\\
&= C(h^2+\dt)^\beta(1+|\log(h^2+\dt)|)\\
&\leq C(h^{2\beta}+(\dt)^\beta)|\log(h^2+\dt)|.
\end{aligned}
\end{equation}
for $h^2+\dt\leq 1/e$, where $C>0$ depends on $T$.
The combination of \eqref{eq:proofSHE1}, \eqref{eq:proofSHE2}, \eqref{eq:proofSHE3} and \eqref{eq:proofSHE4} finishes the proof.
\end{proof}
\subsection{Stochastic Volterra integro-differential equations}\label{subsec:sve}
Here we consider a stochastic integro-differential equation of Volterra-type where the deterministic equation exhibits a parabolic character. The error analysis is basically analogous to the heat equation and therefore we skip some computational details.
\label{rem:volterra} The weak solution of the Volterra-type stochastic evolution equation, a simple model of viscoelastic materials in the presence of noise,
$$
\dl X(t) + \left ( \int_0^t \frac{1}{\Gamma(\rho-1)}(t-s)^{\rho-2} \Lambda X(s) \, \dl s \right ) \, \dl t = \dl L(t),~t\in (0,T]; ~ X(0) =X_0\in H,
$$
is also given by \eqref{eq:X}, where $(E(t))_{t\ge 0}$ is the solution operator of the linear, homogeneous deterministic problem (see, for example, \cite{BGK15,CDaPP,Sperlich} for Gaussian and fractional Brownian noise) and $\rho\in (1,2)$. Because of the special convolution kernel above, the equation can be also viewed as a fractional-in-time differential equation.
Using the same finite element approximation in space as for the heat equation and a convolution quadrature in time we consider the following recurrence (see \cite{KovPri14a,KovPri14b} for the Gaussian case),
\begin{equation} \label{eq:full_scheme volterra}
X^n_{h,\Delta t} - X^{n-1}_{h,\Delta t} + \Delta t \left ( \sum_{k=1}^{n} \omega_{n-k}\,  \Lambda_{h} X^k_{h,\Delta t} \right ) = P_h(L(t_n)-L(t_{n-1})), \quad n\geq 1,
\end{equation}
with $X^0_{h,\Delta t}=P_hX_0$ and convolution weights $(\omega_k)_{k\ge 0}$ chosen according to (see \cite{lubich88,lubich88II})
\begin{equation} \label{eq:weight}
\left ( \frac{1 - z}{\Delta t} \right )^{1-\rho} = \sum_{k\geq 0} \omega_k z^k, \quad |z|<1.
\end{equation}
The solution to \eqref{eq:full_scheme volterra} can again be written in the form \eqref{eq:schemeSHE2} with a suitable operator family $(E^n_{h,\Delta t})_{n\in \mathbb{N}}$, (see \cite{KovPri14a,KovPri14b} for the Gaussian case). Then, define $(\tilde{E}(t))_{t\in [0,T]}$ and $(\tilde{X}(t))_{t\in [0,T]}$ according to \eqref{eq:EtildeSHE} and \eqref{eq:XtildeSHE}, respectively. In contrast to the heat equation where \eqref{eq:detEstSHE1} and \eqref{eq:detEstSHE2} hold, one has the deterministic estimates (see \cite{Lubich_et_al1996} and \cite[Theorem 3.1]{KovPri14b})
\begin{align}
\nnrm{\tilde E(s)-E(s)}{\bo(H)}&\leq C(h^{2/\rho}+\dt)s^{-1},\label{eq:detEstSVE1}\\
\nnrm{\Lambda^\alpha E(s)}{\bo(H)}+\nnrm{\Lambda^\alpha \tilde E(s)}{\bo(H)}&\leq C s^{-\rho\alpha},\quad 0\leq\alpha\leq 1/{(2\rho)},\label{eq:detEstSVE2}
\end{align}
$s\in(0,T]$, where $C>0$ does not depend on $h$, $\dt$ and $s$. Note further, that if $\nnrm{\Lambda^{-1/(2\rho)}Q^{1/2}}{\hs( H)}<\infty$, then using It\^o's isometry, we get the stability estimate as in the Gaussian case, see \cite[page 2333]{KovPri14a},
$$
\nnrm{X(T)}{L^2(\Omega,H)}\le C (\nnrm{\Lambda^{-1/(2\rho)}Q^{1/2}}{\hs( H)}+\nnrm{X_0}{L^2(\Omega,H)}),
$$
and, in particular, Assumption \ref{ass:abstractSetting} (iv) holds. Furthermore, for $t\in[\epsilon, T]$ and $x\in H$, we have that
\begin{equation*}
\|E(t)x\|_{H}=\|\Lambda^{1/(2\rho)}E(t)\Lambda^{-1/(2\rho)}x\|_H\le \|\Lambda^{1/(2\rho)}E(t)\|_{\bo(H)}\|\,\|\Lambda^{-1/(2\rho)}x\|_{H}\le \|\Phi_{\epsilon}x\|_H.
\end{equation*}
Therefore, Assumption \ref{ass:abstractSetting} (v) holds with $\Phi_{\epsilon}:=\sup_{t\in [\epsilon, T]}\|\Lambda^{1/(2\rho)}E(t)\|_{\mathcal{L}(H)}\Lambda^{-1/(2\rho)}$, where the supremum is finite because of \eqref{eq:detEstSVE2}.
Hence, via an analogous calculation as for the heat equation above, setting $H=U=L^2(\dom)$, $B=\id H$, assuming $\|\Lambda^{\frac{\beta-1/\rho}{2}}Q^\frac12\|^2_{\hs(H)}<\infty$, $\beta\in (0,1/\rho)$, $X_0\in L^2(\Omega,\cF_0,\bP;H)$, and using \eqref{eq:detEstSVE1} and \eqref{eq:detEstSVE2}, one arrives at the weak error estimate
\[\big|\bE\big(g(X^N_{h,\dt})-g(X(T))\big)\big|\leq C(h^{2\beta}+(\dt)^{\rho\beta})|\log(h^{2/\rho}+\dt)|, \]
for $h^{2/\rho}+ \dt\leq 1/e$. This is essentially twice the strong rate where the latter is the same as in the Gaussian case \cite{KovPri14a}, as the strong error analysis carries over to our setting, cf.~Remark \ref{rem:strongErrorSHE}.

\section{Application to the wave equation}

\label{sec:Wave}

Here, we apply the general error representation from Section~\ref{sec:An error representation formula} to a discretization of the stochastic wave equation~\eqref{eq:SWE}.

Let $\dom\subset\bR^d$ be a convex bounded domain and let the spaces $\dot H^\alpha$, $\alpha\in\bR$, be as in Section~\ref{sec:Heat}. We use the product spaces
\begin{equation*}
\cH^\alpha:=\dot H^\alpha\times\dot H^{\alpha-1},\quad\alpha\in\bR,
\end{equation*}
with inner product $\langle v,w\rangle_{\cH^\alpha}:=\langle v_1,w_1\rangle_{\alpha}+\langle v_2,w_2\rangle_{\alpha-1}$, $v=(v_1,v_2)^\top$, $w=(w_1,w_2)^\top$ and norm $\nnrm{v}{\cH^\alpha}=(|v_1|_{\alpha}^2+|v_2|_{\alpha-1}^2)^{1/2}$, where $\langle\cdot,\cdot\rangle_\alpha$ and  $\langle\cdot,\cdot\rangle_{\alpha-1}$ are the inner products in $\dot H^\alpha$ and $\dot H^{\alpha-1}$ corresponding to the norms $|\cdot|_\alpha$ and $|\cdot|_{\alpha-1}$ introduced in Section~\ref{sec:Heat}. We set
\[ H:=\cH^0=\dot H^0\times\dot H^{-1}=L^2(\dom)\times H^{-1}(\dom),\quad  U:=\dot H^0=L^2(\dom)\] and define operators $A:D(A)\subset  H\to H$ and $B\in\bo( U, H)$ by setting $D(A):= \cH^1$ and
\begin{equation*}
A:=\begin{pmatrix}0 & -I \\ \Lambda & 0\end{pmatrix},\quad B:=\begin{pmatrix}0\\I\end{pmatrix},
\end{equation*}
where the Laplace operator
$\Lambda$ from Section~\ref{sec:Heat} is now considered as an operator from $\dot H^1$ to $\dot H^{-1}$. It is well-known that $-A$ generates a strongly continuous semigroup $(E(t))_{t\geq0}\subset\bo( H)$ given by
\begin{equation}\label{eq:E(t)SWE}
E(t)=\begin{pmatrix}C(t) & \Lambda^{-1/2}S(t) \\ -\Lambda^{1/2}S(t) & C(t) \end{pmatrix},
\end{equation}
where $C(t):=\cos(t\Lambda^{1/2})$ and $S(t):=\sin(t\Lambda^{1/2})$ are the cosine and sine operators; compare \cite[Example~B.1]{PesZab07}, \cite[Section~A.5.4]{DPZ92} and \cite[Section~3.14]{ABHN11}.

With these definitions the abstract equation \eqref{eq:mainEq} becomes the stochastic wave equation~\eqref{eq:SWE} with $ H$-valued solution $(X(t))_{t\geq0}=((X_1(t),X_2(t))^\top)_{t\geq0}$. As in the Gaussian case, cf.~\cite[Lemma~4.1]{KovLarLin13}, one sees that the condition $\nnrm{\Lambda^{-1/2}Q^{1/2}}{\hs(\dot H^0)}<\infty$ implies $\eqref{eq:assEB1}$ and hence the existence of a unique weak solution $X=(X(t))_{t\geq 0}$, given that the initial condition $X_0=(X_{0,1},X_{0,2})^\top$ is $ H$-valued and $\cF_0$-measurable. Furthermore,
\begin{equation}\label{eq:wstab}
\nnrm{X(T)}{L^2(\Omega,H)}\le T\nnrm{\Lambda^{-1/2}Q^{1/2}}{\hs(\dot H^0)}+\nnrm{X_0}{L^2(\Omega,H)}.
\end{equation}

The discretization of Eq.~\eqref{eq:SWE} is done via finite elements
in space and an $I$-stable rational single step scheme of order in time. (By `$I$-stable' we mean what is called `$I$-acceptable' in \cite{NorWan79}.) We use the finite element setting introduced in Section~\ref{sec:Heat}, the only difference being that we assume a possibly higher order approximation property of the Ritz projection $\Pi_h$ of the form
\begin{equation}\label{eq:basicFEMestimate2}
\nnrm{\Pi_hv-v}{L^2(\dom)}\leq C h^\beta|v|_\beta,\quad v\in\dot H^\beta,\;1\leq\beta\leq r,
\end{equation}
with $r\ge 2$. For $r=2$ this requires no further assumption on the domain $\dom$ (other than convexity) and $S_h$ can be chosen to be the space of continuous piecewise linear functions on a triangulation of $\dom$ with maximal mesh-size $h$, as in the case of the heat equation. For $r>2$ this firstly requires extra assumptions on the domain $\dom$, namely small enough interior angles in case of a polygon, or smooth enough boundary in case of a curved boundary. The reason is that to achieve $r>2$ one needs an elliptic regularity estimate $\|u\|_{H^r}\le C\|\Lambda u\|_{H^{r-2}}$ with $r>2$ instead of \eqref{eq:ellreg} corresponding to $r=2$. Furthermore, if the boundary is curved, then the triangulation is not exact and one has to be more precise while approximating near the boundary and hence one needs special elements. For example, \eqref{eq:basicFEMestimate2} holds for $r=4$ for bounded convex domains with smooth boundary and $S_h$ consisting of continuous piecewise cubic polynomials using special, so-called isoparametric elements, near the boundary, see \cite[Chapter 1]{Tho06}.
Although \eqref{eq:basicFEMestimate2} does not appear explicitly in the present paper, it is the key ingredient in the proof of the deterministic error estimate for the finite element approximation of the wave equation which we will use later on and hence we state it as our abstract assumption on the finite element spaces $S_h$.
Let the discretization $A_h:S_h\times S_h\to S_h\times S_h$ of the operator $A:D(A)\subset H\to H$ be defined by
\begin{equation*}
A_h:=\begin{pmatrix}0 & -I \\ \Lambda_h & 0\end{pmatrix},
\end{equation*}
where $\Lambda_h:S_h\to S_h$ is the discrete Laplacian introduced in \eqref{eq:discreteLaplace}.
Then $-A_h$ generates a strongly continuous semigroup $(E_h(t))_{t\geq0}\subset \bo(S_h\times S_h)$. As in Section~\ref{sec:Heat}, we consider for $N\in\bN$ a uniform grid $t_n=n\dt=n (T/N)$, $n=0,\ldots,N$, on a finite time interval $[0,T]$. We approximate the operators $E_h(t_n)\in\bo(S_h\times S_h)$ by
\[E_{h,\dt}^n:=(R(\dt A_h))^n,\]
where $R$ is a rational function that satisfies the approximation and stability properties
\begin{equation*}
\begin{aligned}
|R(iy)-e^{-iy}|&\leq C |y|^{p+1},\quad |y|\leq b,\\
|R(iy)|&\leq 1,\quad y\in\bR,
\end{aligned}
\end{equation*}
for some positive integer $p$ and some $b>0$; see \cite{BakBra79, BreTho80} for details. For instance, choosing $R(\lambda)=1/(1-\lambda)$ and  $R(\lambda)=(2-\lambda)/(2+\lambda)$ yields the backward Euler method ($p=1$) and the Crank-Nicolson method ($p=2$), respectively.

The numerical scheme for the stochastic wave equation \eqref{eq:mainEq} can now be formulated as follows:
For $h>0$ and $N\in\bN$, the discretization $(X^n_{h,\dt})_{n=0,\ldots,N}$ of $(X(t))_{t\in[0,T]}$ in space and time is given as the solution to
\begin{equation}\label{eq:schemeSWE1}
X_{h,\dt}^n=E_{h,\dt}\big(X_{h,\dt}^{n-1}+P_h B(L(t_n)-L(t_{n-1}))\big),\quad n=1,\ldots,N;\quad X_{h,\dt}^0=P_hX_0.
\end{equation}
By slight abuse of notation, we denote here and in the sequel by $P_h$ both the generalized $L^2$-projection from $\dot H^{-1}$ onto $S_h$ defined by $\langle P_hv,w\rangle_{L^2(\dom)}=\langle v,w\rangle_{\dot H^{-1}\times\dot H^1}$, $v\in \dot H^{-1}$, $w\in S_h$, and the corresponding projection from $ H=\dot H^{0}\times\dot H^{-1}$ onto $S_h\times S_h$ defined by the action of the former projection on the coordinates of elements in $\dot H^{0}\times\dot H^{-1}$. Moreover, $P^1: H\to\dot H^0$ is the projection of elements in $H=\dot H^0\times\dot H^{-1}$ on the first coordinate.

\begin{remark}[strong error]\label{rem:strongErrorSWE}
As observed for the discretization of the heat equation in Remark~\ref{rem:strongErrorSHE}, strong $L^2$-error estimates for the scheme \eqref{eq:schemeSWE1} carry over from the Gaussian case in the L\'{e}vy $L^2$-martingale case since they only use It\^{o}'s isometry \eqref{eq:ItoIsom}. Arguing as in the proof of \cite[Theorem~4.13]{KovLarLin13}, we obtain that, if
\begin{equation}\label{eq:SWEassQ}
\nnrm{\Lambda^{\frac{\beta-1}2}Q^{\frac12}}{\hs(\dot H^0)}<\infty\quad\text{ and }\quad X_0\in L^2(\Omega,\cF_0,\bP; \cH^{\beta})
\end{equation}
for some $\beta>0$, then the scheme \eqref{eq:schemeSWE1} approximates the first component $X_1=P^1X$ of the solution $X$ to \eqref{eq:mainEq} with strong order $\min(\beta r/(r+1),r)$ in space and $\min(\beta p/(p+1),1)$ in time:
\[\nnrm{X^n_{h,\dt,1}-X_1(t_n)}{L^2(\Omega;\cdot H^0)}\leq C\big(h^{\min(\beta\frac{r}{r+1},r)}+(\dt)^{\min(\beta\frac{p}{p+1},1)}\big),\quad n=0,\ldots,N.\]
Here we have set $X^n_{h,\dt,1}:=P^1X^n_{h,\dt}$. The condition \eqref{eq:SWEassQ} implies that the solution $X=(X(t))_{t\geq0}$ takes values in $ \cH^\beta$, cf.~\cite[Theorem~3.1]{KovLarSae10}.

\end{remark}

The solution to the scheme \eqref{eq:schemeSWE1} is given by
\begin{equation*}\label{eq:schemeSWE2}
X^n_{h,\dt}=E_{h,\dt}^nP_hX_0+\sum_{j=1}^n E^{n-j+1}_{h,\dt}P_hB(L(t_n)-L(t_{n-1})),\quad n=0,\ldots,N.
\end{equation*}
For $t\in[0,T]$, define operators $\tilde E(t)=\tilde E_{h,\dt}(t)\in\bo( H)$ by
\begin{equation}\label{eq:EtildeSWE}
\tilde E(t)=\tilde E_{h,\dt}(t):=\one_{\{0\}}(t)P_h+\sum_{j=1}^N\one_{(t_{n-1},t_n]}(t)E_{h,\dt}^nP_h,
\end{equation}
where the projection $P_h$ is understood as a mapping from $ H=\dot H^{0}\times\dot H^{-1}$ to $S_h\times S_h$. Then, analogously to the corresponding argument in Section~\ref{sec:Heat}, one sees that the $S_h\times S_h$-valued process $(\tilde X(t))_{t\in[0,T]}=(\tilde X_{h,\dt}(t))_{t\in[0,T]}$ defined by
\begin{equation}\label{eq:XtildeSWE}
\tilde X(t)=\tilde X_{h,\dt}(t):=\tilde E_{h,\dt}(t)X_0+\int_0^t\tilde E_{h,\dt}(t-s)B\,\dl L(s)
\end{equation}
satisfies $X^n_{h,\dt}=\tilde X(t_n)$ $\bP$-almost surely.

The proof of the deterministic error estimate in the next lemma is postponed to the end of this section.
\begin{lemma} \label{lem:SWE}
Let $\alpha\geq0$. The operators $E(t)$ and $\tilde E(t)=\tilde E_{h,\dt}(t)$ defined in \eqref{eq:E(t)SWE} and \eqref{eq:EtildeSWE} satisfy the error estimate
\begin{equation}\label{eq:detEstSWE}
\begin{gathered}
    \sup_{t\in[0,T]}\big(\nnrm{P^1(\tilde E(t)-E(t))}{\bo( \cH^\alpha,\dot H^0)}
    +\nnrm{P^1(\tilde E(t)-E(t))B}{\bo(\dot H^{(\alpha/2)-1},\dot H^{-\alpha/2})}\big)\\
    \leq C\big(h^{\min(\alpha\frac{r}{r+1},r)}+(\dt)^{\min(\alpha\frac{p}{p+1},1)}\big),
\end{gathered}
\end{equation}
for $\dt\leq1$, where $C=C(T)>0$ does not depend on $h$ and $\dt$.
\end{lemma}

We are now in the position to prove the following result concerning the weak error of the approximation $X^N_{h,\dt,1}:=P^1X^N_{h,\dt}$ of the first component $X_1(T)=P^1X(T)$ of the solution to the stochastic wave eqation~\eqref{eq:mainEq} at time $T$.

\begin{theorem}\label{thm:SWE}
Let $X_0\in L^2(\Omega,\cF_0,\bP; \cH^{2\beta})$ for some $\beta>0$ and $g\in C^2(H,\bR)$ with $\sup_{x\in H}\nnrm{g''(x)}{\bo(H)}<\infty$. Suppose that either of the following conditions holds.
\begin{itemize}
\item[\upshape (i)] $\nnrm{\Lambda^{(\beta-1)/2}Q^{1/2}}{\hs(\dot H^0)}<\infty$ and
\begin{equation}\label{eq:assgSWE}
\sup_{x\in\dot H^0}\nnrm{\Lambda^{\frac\beta2}g''(x)\Lambda^{-\frac\beta2}}{\bo(\dot H^0)}<\infty.
\end{equation}
\item[\upshape (ii)]
\begin{equation}\label{eq:weqii}
\lim_{m\to\infty}\int_{U_1}\|\Lambda^{-\frac12}p_m y\|_{\dot{H}^0}\|\Lambda^{\beta-\frac{1}{2}}p_m y\|_{\dot{H}^0}\,\nu(\dl y)<\infty,
\end{equation}
where $p_m$ denotes the orthogonal projection from $\dot H^0$ to $\operatorname{span}\{\varphi_1,\ldots,\varphi_m\}$, $(\varphi_k)_{k\in\mathbb N}$ being an orthonormal basis of $\dot H^0$ consisting of eigenvectors of $\Lambda$.
\end{itemize}
Then, there is a unique weak solution $(X(t))_{t\geq0}$ to Eq.~\eqref{eq:mainEq} given by \eqref{eq:X}. 

Let $(X_{h,\dt}^n)_{n=0,\ldots,N}$ be given by the scheme \eqref{eq:schemeSWE1}. Then, there exists a constant $C=C(g,T)>0$ that does not depend on $h$ and $\dt$, such that for $\dt\leq 1$
$$
    \big|\bE\big(g(X^N_{h,\dt,1})-g(X_1(T))\big)\big|\leq C\big(h^{\min(2\beta\frac{r}{r+1},r)}+(\dt)^{\min(2\beta\frac{p}{p+1},1)}\big)
$$
\end{theorem}

Before proving Theorem~\ref{thm:SWE}, we state two remarks and discuss some examples where the conditions of Theorem \ref{thm:SWE}, in particular \eqref{eq:assgSWE} and \eqref{eq:weqii}, are satisfied.

\begin{remark}
As a consequence of Lemma~\ref{lem:integrandIsomorphism} and the fact that $\Lambda^\alpha p_m\in\hs(\dot H^0)$ for all $\alpha\in\bR$ and $m\in\bN$, the terms $\Lambda^{-1/2}p_m y$ and $\Lambda^{\beta-1/2}p_m y$ in \eqref{eq:weqii} are defined in an $L^2(U_1,\nu(\dl y);\dot H^0)$-sense. The sequence $\big(\|\Lambda^{-1/2}p_m y\|_{\dot{H}^0}\|\Lambda^{\beta-1/2}p_m y\|_{\dot{H}^0}\big)_{m\in\bN}$ is monotonically increasing for $\nu$-almost all $y\in U_1$, so that the limit in \eqref{eq:weqii} is in fact a supremum. Moreover, if we explicitly choose $U_1=\dot H^{\beta-1}$ as the state space of $L$, then the condition~(ii) is equivalent to assuming that $\operatorname{supp}\nu\subset \dot H^{2\beta-1}$ and
\[\int_{\dot H^{2\beta-1}}\|\Lambda^{-\frac12}y\|_{\dot{H}^0}\|\Lambda^{\beta-\frac{1}{2}}y\|_{\dot{H}^0}\,\nu(\dl y)<\infty .\]
This choice of $U_1$ is possible w.l.o.g.\ whenever $\nnrm{\Lambda^{(\beta-1)/2}Q^{1/2}}{\hs(\dot H^0)}<\infty$, since then the natural embedding of $U_0=Q^{1/2}U=Q^{1/2}\dot H^0$ into $\dot H^{\beta-1}$ is Hilbert-Schmidt and we can re-expand $L$ in the form \eqref{eq:LPexpansion} as an $\dot H^{\beta-1}$-valued martingale, compare Remark~\ref{rem:approxL}. However, in the spirit of, e.g., \cite{AppRie10,Rie12,Rie14}, we prefer a formulation of our results that is independent of the specific choice of the state space $U_1$.
\end{remark}

\begin{remark}
Instead of the symmetric condition $\nnrm{\Lambda^{(\beta-1)/2}Q^{1/2}}{\hs(\dot H^0)}<\infty$, the sufficient asymmetric condition $\nnrm{\Lambda^{\beta-1/2}Q\Lambda^{-1/2}}{\nuc(\dot{H}^0)}<\infty$ is imposed in \cite{KovLarLin13} in the Wiener case in order to double the rate of strong convergence for the wave equation.
The asymmetric condition \eqref{eq:weqii}
appearing in (ii) above, which is again sufficient for\linebreak $\nnrm{\Lambda^{(\beta-1)/2}Q^{1/2}}{\hs(\dot H^0)}<\infty$,
resembles the same situation in the present case.
\end{remark}


\begin{example}
An important basic example which satisfies \eqref{eq:assgSWE} is $g(x)=\|x\|^2_{\dot{H}^0}$.
\end{example}

\begin{example}
As another example for a test function $g$ satisfying \eqref{eq:assgSWE} consider
\begin{equation*}
g(x):=f(\langle\varphi_1,x\rangle_{\dot H^0},\ldots,\langle\varphi_n,x\rangle_{\dot H^0}),\quad x\in \dot H^0,
\end{equation*}
where $f\in C^2(\bR^n,\bR)$ has bounded second order derivatives and $(\varphi_k)_{k\in\bN}\subset D(\Lambda)$ is an orthonormal basis of $\dot H^0=L^2(\dom)$ consisting of eigenfunctions of $\Lambda $ with corresponding eigenvalues $(\lambda_k)_{k\in\bN}\subset(0,\infty)$. Then, for $x,y\in\dot H^0$,
\[
\Lambda^{\beta/2}g''(x)\Lambda^{-\beta/2}y=\sum_{j,k=1}^n\lambda_j^{-\beta/2}\lambda_k^{\beta/2}(\partial_j\partial_kf)\big(\langle\varphi_1,x\rangle_{\dot H^0},\ldots,\langle\varphi_n,x\rangle_{\dot H^0}\big)\langle\varphi_j,y\rangle_{\dot H^0}\varphi_k
\]
and \eqref{eq:assgSWE} holds. More generally, the condition \eqref{eq:assgSWE} is satisfied by all $g\in C^2(\dot H^0,\bR)$ of the form $g=\tilde g\circ\Lambda^{-\beta/2}$ with $\tilde g\in C^2(\dot H^0,\bR)$ satisfying $\sup_{x\in\dot H^0}\nnrm{\tilde g''(x)}{\bo(\dot H^0)}<\infty$. For such $g$ we have $g''(x)=\Lambda^{-\beta/2}\tilde g''(\Lambda^{-\beta/2}x)\Lambda^{-\beta/2}$.
\end{example}

\begin{example}\label{ex:spec}
Consider the situation of Example \ref{ex:1dlevy}; that is when
\[\nu=\sum_{k\in\bN}\nu_k\circ\pi_k^{-1},\]
where $\nu_k$ is the L\'{e}vy measure of $L_k$ and $\pi_k:\bR\to U_1$ is defined by $\pi_k(\xi):=\xi e_k$. Let us assume that $e_k=\sqrt{q_k}\,\varphi_k$, where $(q_k)_{k\in\bN}$ is a bounded sequence of positive numbers and $(\varphi_k)_{k\in \bN}$ is an orthonormal basis of $\dot{H}^0$ consisting of eigenfunctions of $\Lambda$ with corresponding eigenvalues $(\lambda_k)_{k\in \bN}$.
Then, with $p_m$ as in (ii) in Theorem \ref{thm:SWE},
\begin{align*}
&\lim_{m\to\infty}\int_{U_1}\|\Lambda^{-\frac12}p_m y\|_{\dot{H}^0}\|\Lambda^{\beta-\frac{1}{2}}p_m y\|_{\dot{H}^0}\,\nu(\dl y)\\
&=\sum_{k\in \bN}\int_{\bR}\xi^2q_k\|\Lambda^{-\frac12}\varphi_k\|_{\dot{H}^0}\|\Lambda^{\beta-\frac12}\varphi_k\|_{\dot{H}^0}\,\nu_k(\dl \xi)\\
&=\sum_{k\in \bN}\int_{\bR}\xi^2q_k\lambda_k^{-\frac12}\|\varphi_k\|_{\dot{H}^0}\lambda_k^{\beta-\frac12}\|\varphi_k\|_{\dot{H}^0}\,\nu_k(\dl \xi)\\
&=\sum_{k\in \bN}\int_{\bR}\xi^2q_k\|\Lambda^{\frac{\beta-1}2}\varphi_k\|_{\dot{H}^0}\|\Lambda^{\frac{\beta-1}2}\varphi_k\|_{\dot{H}^0}\,\nu_k(\dl \xi)\\
&=\lim_{m\to\infty}\int_{U_1}\|\Lambda^{\frac{\beta-1}2}p_my\|_{\dot{H}^0}\|\Lambda^{\frac{\beta-1}2}p_my\|_{\dot{H}^0}\,\nu(\dl y)
\,=\,\nnrm{\Lambda^{\frac{\beta-1}2}Q^{\frac12}}{\hs(\dot H^0)}^2,
\end{align*}
where we used Lemma~\ref{lem:integrandIsomorphism} in the last step.
That is, when $\nu$ is concentrated on the eigenspaces $\{r\varphi_k:r\in\bR\}$, $k\in\bN$, of $\Lambda$, then the abstract asymmetric condition \eqref{eq:weqii} coincides with the familiar symmetric Hilbert-Schmidt condition. The situation is similar in the Wiener case \cite{KovLarLin13} when $\Lambda$ and $Q$ commute.
\end{example}

\begin{proof}[Proof of Theorem~\ref{thm:SWE}]
First suppose that (i) holds. We apply Theorem~\ref{thm:errorRep} and Corollary~\ref{cor:errorRep} with $G=g\circ P^1$. Note that $G'(x)=(P^1)^*g'(P^1x)\in H$ and $$G''(x)=(P^1)^*g''(P^1x)P^1\in\bo( H)$$ for all $x\in H$, where $(P^1)^*\in\bo(\dot H^0, H)$ is the Hilbert space adjoint of $P^1\in\bo( H,\dot H^0)$. Using \eqref{eq:u_xu_xx}
one obtains
\begin{equation}\label{eq:proofSWE0}
u_x(t,\xi)=\bE\big((P^1)^*g'(P^1Z(T,t,x))\big)\big|_{x=\xi},\quad u_{xx}(t,\xi)=\bE\big((P^1)^*g''(P^1Z(T,t,x))P^1\big)\big|_{x=\xi}
\end{equation}
for all $H$-valued random variables $\xi$ and $t\in[0,T]$.


We combine \eqref{eq:wstab}, \eqref{eq:proofSWE0} and the deterministic error estimate \eqref{eq:detEstSWE} with $\alpha=2\beta$ in order to estimate the first term on the right hand side of the error representation formula~\eqref{eq:errorRep4} in Corollary~\ref{cor:errorRep}. We have, where the first inequality follows similarly as for the stochastic heat equation, that
\begin{equation}\label{eq:proofSWE1}
\begin{aligned}
\big|\bE\big\{&u(0,\tilde E(T)X_0)-u(0,E(T)X_0)\big\}\big|\\
&\leq \int_0^1\big\|g'\big(P^1Z\big(T,0,Y(0)+\theta(\tilde Y(0)-Y(0))\big)\big)\big\|_{L^2(\Omega, \dot{H}^0)}\,\dl\theta\\
&\qquad \times \|P^1(\tilde E(T)-E(T))X_0\|_{L^2(\Omega, \dot{H}^0)}\\
&\leq C\big(1+\int_0^1\big\|Z\big(T,0,Y(0)+\theta(\tilde Y(0)-Y(0))\big)\big\|_{L^2(\Omega, H)}\,\dl \theta\big)\\
&\qquad \times \nnrm{P^1(\tilde E(T)-E(T))}{\bo( \cH^{2\beta},\dot H^0)}\,\nnrm{X_0}{L^2(\Omega; \cH^{2\beta})}\\
&\leq C\big(1+\nnrm{\Lambda^{-1/2}Q^{1/2}}{\hs(\dot H^0)}+\nnrm{X_0}{L^2(\Omega; H)}\big)\,\nnrm{X_0}{L^2(\Omega; \cH^{2\beta})}\\
&\qquad \times \big(h^{\min(2\beta\frac{r}{r+1},r)}+(\dt)^{\min(2\beta\frac{p}{p+1},1)}\big).
\end{aligned}
\end{equation}
Using \eqref{eq:proofSWE0}, Lemma~\ref{lem:integrandIsomorphism} and Remark~\ref{not:Phix}, the integral of the function $\Psi_1$ in the second term on the right hand side of the formula \eqref{eq:errorRep4} can be treated as follows:
\begin{equation}\label{eq:proofSWE2}
\begin{aligned}
\Big|\bE&\int_0^T\int_{ U_1}\int_0^1\Psi_1(t,\theta,y)\,\dl\theta\,\nu(\dl y)\,\dl t\Big|\\
&= \Big|\bE\int_0^T\int_{ U_1}\int_0^1(1-\theta)\Big\langle\bE\big(g''\big(P^1Z(T,t,x+E(T-t)By+\theta F(T-t)y)\big)\big)\big|_{x=\tilde Y(t)}\\
&\quad\times P^1F(T-t)y,\;P^1F(T-t)y\Big\rangle_{\dot H^0}\dl\theta\,\nu(\dl y)\,\dl t\Big|\\
&\leq \sup_{x\in\dot H^0}\nnrm{g''(x)}{\bo(\dot H^0)}\int_0^T\nnrm{P^1F(T-t)}{\hs( U_0,\dot H^0)}^2\dl t\\
&\leq \sup_{x\in\dot H^0}\nnrm{g''(x)}{\bo(\dot H^0)} C \big(h^{\min(\beta\frac{r}{r+1},r)}+(\dt)^{\min(\beta\frac{p}{p+1},1)}\big)^2
\end{aligned}.
\end{equation}
The last step in \eqref{eq:proofSWE2} is due to the fact that, by It\^{o}'s isometry \eqref{eq:ItoIsom}, the integral in the penultimate line is the square of the strong error $\nnrm{X^N_{h,k,1}-X_1(T)}{L^2(\Omega;\dot H^0)}$ for zero initial condition $X_0=0$; it can be estimated as in the Gaussian case \cite[Theorem~4.13]{KovLarLin13}, compare Remark~\ref{rem:strongErrorSWE}.

Concerning the integral of the function $\Psi_2$ in the second term on the right hand side of Eq.~\eqref{eq:errorRep4}, we have by \eqref{eq:proofSWE0}, Lemma~\ref{lem:integrandIsomorphism}, \eqref{eq:estHS} and since $U_0=Q^{1/2}(U)=Q^{1/2}(\dot H^0)$,
\begin{equation}\label{eq:proofSWE3}
\begin{aligned}
\Big|\bE&\int_0^T\int_{ U_1}\int_0^1\Psi_2(t,\theta,y)\,\dl\theta\,\nu(\dl y)\,\dl t\Big|\\
&= \Big|\bE\int_0^T\int_{ U_1}\int_0^1\Big\langle\bE\big(g''\big(P^1Z(T,t,x+\theta E(T-t)By)\big)\big)\big|_{x=\tilde Y(t)}\\
&\quad\times P^1E(T-t)By,P^1F(T-t)y\Big\rangle_{\dot H^0}\dl\theta\,\nu(\dl y)\,\dl t\Big|\\
&= \Big|\bE\int_0^T\int_{ U_1}\int_0^1\Big\langle\bE\big(\Lambda^{\frac\beta2}g''\big(P^1Z(T,t,x+\theta E(T-t)By)\big)\Lambda^{-\frac\beta2}\big)\big|_{x=\tilde Y(t)}\\
&\quad\times \Lambda^{\frac\beta2}P^1E(T-t)B\Lambda^{\frac{1-\beta}2}\Lambda^{\frac{\beta-1}2}y,\;\Lambda^{-\frac\beta2}P^1F(T-t)\Lambda^{\frac{1-\beta}2}\Lambda^{\frac{\beta-1}2}y\Big\rangle_{\dot H^0}\dl\theta\,\nu(\dl y)\,\dl t\Big|\\
&\leq \sup_{x\in\dot H^0}\nnrm{\Lambda^{\frac\beta2}g''(x)\Lambda^{-\frac\beta2}}{\bo(\dot H^0)}\nnrm{\Lambda^{\frac{\beta-1}2}Q^{\frac12}}{\hs(\dot H^0)}^2\\
&\quad\times \int_0^T\nnrm{\Lambda^{\frac\beta2}P^1E(T-t)B\Lambda^{\frac{1-\beta}2}}{\bo(\dot H^0)}\nnrm{\Lambda^{-\frac\beta2}P^1F(T-t)\Lambda^{\frac{1-\beta}2}}{\bo(\dot H^0)}\,\dl t.
\end{aligned}
\end{equation}
Note that, by the definition of $B = (0,I)^\top$ and $E(t)$ from \eqref{eq:E(t)SWE} we have
\begin{equation}\label{eq:etb}
\nnrm{\Lambda^{\frac\beta2}P^1E(T-t)B\Lambda^{\frac{1-\beta}2}}{\bo(\dot H^0)}
=\nnrm{\Lambda^{\frac{\beta-1}2}S(T-t)\Lambda^{\frac{1-\beta}2}}{\bo(\dot H^0)}
=\nnrm{S(T-t)}{\bo(\dot H^0)}\leq1;
\end{equation}
it remains to estimate the integral
\begin{equation*}\label{eq:proofSWE4}
\begin{aligned}
\int_0^T\nnrm{\Lambda^{-\frac\beta2}P^1F(T-t)\Lambda^{\frac{1-\beta}2}}{\bo(\dot H^0)}\dl t
&=\int_0^T\nnrm{P^1F(t)}{\bo(\dot H^{\beta-1},\dot H^{-\beta})}\dl t\\
&=\int_0^T\nnrm{P^1(\tilde E(t)-E(t))B}{\bo(\dot H^{\beta-1},\dot H^{-\beta})}\dl t.
\end{aligned}
\end{equation*}
To this end, it suffices to apply the deterministic error estimate \eqref{eq:detEstSWE} with $\alpha=2\beta$. The combination of \eqref{eq:proofSWE1}, \eqref{eq:proofSWE2} and \eqref{eq:proofSWE3} finishes the proof.

Next, suppose that (ii) holds.
By Lemma~\ref{lem:integrandIsomorphism} we have
\begin{align*}
\nnrm{p_m\Lambda^{(\beta-1)/2}Q^{1/2}}{\hs(\dot H^0)}^2
&=\nnrm{\Lambda^{(\beta-1)/2}p_mQ^{1/2}}{\hs(\dot H^0)}^2
=\int_{U_1}\|\Lambda^{(\beta-1)/2}p_my\|_{\dot{H}^0}^2\,\nu(\dl y)\\
&=\int_{U_1}\langle\Lambda^{-1/2}p_my,\Lambda^{\beta-1/2}p_my\rangle_{\dot{H}^0}\,\nu(\dl y)\\
&\leq\int_{U_1}\|\Lambda^{-1/2}p_my\|_{\dot{H}^0}\|\Lambda^{\beta-1/2}p_my\|_{\dot{H}^0}\,\nu(\dl y),
\end{align*}
hence,
$$
    \nnrm{\Lambda^{(\beta-1)/2}Q^{1/2}}{\hs(\dot H^0)}^2
    =\lim_{m\to\infty}\nnrm{p_m\Lambda^{(\beta-1)/2}Q^{1/2}}{\hs(\dot H^0)}^2
    < \infty,
$$
proving that there is a unique weak solution $(X(t))_{t\geq0}$ to Eq.~\eqref{eq:mainEq} given by \eqref{eq:X}.
To estimate the weak error, we apply an approximation procedure and consider for $m\in\bN$ the $H$-valued random variables
$X^{[m]}(T):=E(T)X_0+\int_0^TE(T-s)Bp_m\,\dl L(s)$ and $\tilde X^{[m]}(T)=\tilde X^{[m]}_{h,\dt}(T):=\tilde E(T)X_0+\int_0^T\tilde E(T-s)Bp_m\,\dl L(s)$. Using It\^{o}'s isometry and the fact that $\nnrm{(I-p_m)\Lambda^{-1/2}Q^{1/2}}{\hs(\dot H^0)}\to 0$ as $m\to\infty$, we get that $X^{[m]}(T)\xrightarrow{m\to\infty} X(T)$ and $\tilde X^{[m]}(T)\xrightarrow{m\to\infty}\tilde X(T)$ in $L^2(\Omega;H)$. As a consequence of this and \eqref{eq:growthGG'}, we obtain $e^{[m]}(T):=\bE\big( G(\tilde X^{[m]}(T))-G(X^{[m]}(T))\big)\xrightarrow{m\to\infty}  e(T)$. Thus, it suffices to show the desired decay rate for the error $e^{[m]}(T)$ with a constant that does not depend on $m$. To this end, we observe that the estimate \eqref{eq:proofSWE1} can be used without any change, and that the analogue to the estimate \eqref{eq:proofSWE2} gives indeed the desired rate if we use that $\nnrm{p_m\Lambda^{(\beta-1)/2}Q^{1/2}}{\hs(\dot H^0)}\leq\nnrm{\Lambda^{(\beta-1)/2}Q^{1/2}}{\hs(\dot H^0)}<\infty$. Finally, the estimate corresponding to \eqref{eq:proofSWE3} reads

\begin{equation*}\label{eq:proofSWE3b}
\begin{aligned}
\Big|&\bE\int_0^T\int_{U_1}\int_0^1\Big\langle\bE\big(g''\big(P^1Z^{[m]}(T,t,x+\theta E(T-t)Bp_my)\big)\big)\big|_{x=\tilde Y^{[m]}(t)}\\
&\quad\times P^1E(T-t)Bp_my,P^1F(T-t)p_my\Big\rangle_{\dot H^0}\dl\theta\,\nu(\dl y)\,\dl t\Big|\\
&= \Big|\bE\int_0^T\int_{U_1}\int_0^1\Big\langle\bE\big(g''\big(P^1Z^{[m]}(T,t,x+\theta E(T-t)Bp_my)\big)\big)\big|_{x=\tilde Y^{[m]}(t)}\\
&\quad\times P^1E(T-t)B\Lambda^{1/2}\Lambda^{-1/2}p_my,P^1F(T-t)\Lambda^{1/2-\beta}\Lambda^{\beta-1/2}p_my\Big\rangle_{\dot H^0}\dl\theta\,\nu(\dl y)\,\dl t\Big|\\
&\leq \sup_{x\in\dot H^0}\nnrm{g''(x)}{\bo(\dot H^0)}\int_0^T\nnrm{P^1E(T-t)B\Lambda^{1/2}}{\bo(\dot H^0)}\nnrm{P^1F(T-t)\Lambda^{1/2-\beta}}{\bo(\dot H^0)}\,\dl t\\
&\quad\times \int_{U_1}\|\Lambda^{-1/2}p_my\|_{\dot{H}^0}\|\Lambda^{\beta-1/2}p_my\|_{\dot{H}^0}\,\nu(\dl y),
\end{aligned}
\end{equation*}
where $Z^{[m]}$ and $\tilde Y^{[m]}$ are defined by replacing $B$ by $Bp_m$ in the definitions of $Z$ and $\tilde Y$.
By \eqref{eq:detEstSWE} with $\alpha=2\beta$
and the fact that $\nnrm{B}{\bo(\dot H^{\alpha-1}, \cH^\alpha)}=1$ we have that
\begin{align*}
&\sup_{t\in[0,T]}\nnrm{P^1F(T-t)\Lambda^{1/2-\beta}}{\bo(\dot H^0)}=\sup_{t\in[0,T]}\nnrm{P^1(\tilde E(t)-E(t))B\Lambda^{1/2-\beta}}{\bo(\dot H^0)}\\
&\quad =\sup_{t\in[0,T]}\nnrm{P^1(\tilde E(t)-E(t))B}{\bo(\dot H^{2\beta-1},\dot H^0)}\leq C\big(h^{\min(2\beta\frac{r}{r+1},r)}+(\dt)^{\min(2\beta\frac{p}{p+1},1)}\big).
\end{align*}
Finally, by \eqref{eq:weqii} and \eqref{eq:etb} with $\beta =0$, the proof is complete.
\end{proof}

\begin{proof}[Proof of Lemma~\ref{lem:SWE}]
We use the estimates
\begin{equation}\label{eq:detEstSWE1}
\sup_{n\in\{0,\ldots,N\}}\nnrm{P^1(E^n_{h,\dt}P_h-E(t_n))}{\bo( \cH^\alpha,\dot H^0)}\leq C(T)\big(h^{\min(\alpha\frac{r}{r+1},r)}+(\dt)^{\min(\alpha\frac{p}{p+1},p)}\big)
\end{equation}
and
\begin{equation}\label{eq:detEstSWE2}
\nnrm{E(t)-E(s)}{\bo( \cH^\delta, H)}\leq C|t-s|^\delta,\quad t,\,s\geq0,\;\delta\in[0,1].
\end{equation}
from Corollary~4.11 and Lemma~4.4 in \cite{KovLarLin13}. Corollary~4.11 in \cite{KovLarLin13} is based on an error estimate proved in \cite{BakBra79}.

Because of the `piecewise' definition of $\tilde E(t)$ in \eqref{eq:EtildeSWE}, the combination of \eqref{eq:detEstSWE1} and \eqref{eq:detEstSWE2} gives
\begin{equation}\label{eq:detEstSWE3}
\begin{aligned}
&\sup_{t\in[0,T]}\nnrm{P^1(\tilde E(t)-E(t))}{\bo( \cH^\alpha,\dot H^0)}\\
&\leq \sup_{n\in\{0,\ldots,N\}}\nnrm{P^1(\tilde E(t_n)-E(t_n))}{\bo( \cH^\alpha,\dot H^0)}+\sup_{n\in\{1,\ldots,N\}}\sup_{t\in(t_{n-1},t_n)}\nnrm{E(t_n)-E(t)}{\bo( \cH^\alpha,\cH)}\\
&\leq C(T)\big(h^{\min(\alpha\frac{r}{r+1},r)}+(\dt)^{\min(\alpha\frac{p}{p+1},p)}+(\dt)^{\min(\alpha,1)}\big)\\
&=C(T)\big(h^{\min(\alpha\frac{r}{r+1},r)}+(\dt)^{\min(\alpha\frac{p}{p+1},1)}\big)
\end{aligned}
\end{equation}
for $\dt\leq 1$.
It remains to show that
\begin{equation}\label{eq:detEstSWE4}
\sup_{t\in[0,T]}\nnrm{P^1(\tilde E(t)-E(t))B}{\bo(\dot H^{(\alpha/2)-1},\dot H^{-\alpha/2})}\leq C(T)\big(h^{\min(\alpha\frac{r}{r+1},r)}+(\dt)^{\min(\alpha\frac{p}{p+1},1)}\big).
\end{equation}
To this end, we will prove the estimate
\begin{equation}\label{eq:detEstSWE5}
\sup_{n\in\{0,\ldots,N\}}\nnrm{P^1(\tilde E(t_n)-E(t_n))B}{\bo(\dot H^{(\alpha/2)-1},\dot H^{-\alpha/2})}\leq C(T)\big(h^{\min(\alpha\frac{r}{r+1},r)}+(\dt)^{\min(\alpha\frac{p}{p+1},p)}\big).
\end{equation}
Then, \eqref{eq:detEstSWE4} follows from \eqref{eq:detEstSWE5} and \eqref{eq:detEstSWE2} by estimating analogously to \eqref{eq:detEstSWE3} and using the fact that
\begin{align*}
\nnrm{P^1(E(t_n)-E(t))B}{\bo(\dot H^{(\alpha/2)-1},\dot H^{-\alpha/2})}
&=\nnrm{\Lambda^{-\frac\alpha4}P^1(E(t_n)-E(t))B\Lambda^{\frac12-\frac\alpha4}}{\bo(\dot H^0)}\\
&=\nnrm{P^1(E(t_n)-E(t))B\Lambda^{\frac{1-\alpha}2}}{\bo(\dot H^0)}\\
&\leq \nnrm{P^1(E(t_n)-E(t))}{\bo( \cH^\alpha,\dot H^0)}\nnrm{B\Lambda^{\frac{1-\alpha}2}}{\bo(\dot H^0, \cH^\alpha)},
\end{align*}
where $\nnrm{B\Lambda^{\frac{1-\alpha}2}}{\bo(\dot H^0, \cH^\alpha)}=\nnrm{B}{\bo(\dot H^{\alpha-1}, \cH^\alpha)}=1$.

In order to show \eqref{eq:detEstSWE5}, we distinguish the cases $\alpha>2$ and $0\leq\alpha\leq2$.
For $\alpha>2$ we have by \eqref{eq:detEstSWE1}
\begin{equation}\label{eq:proofSWE6}
\begin{aligned}
\sup_{n\in\{0,\ldots,N\}}&\nnrm{P^1(\tilde E(t_n)-E(t_n))B}{\bo(\dot H^{\alpha-1},\dot H^0)}\\
&\leq\sup_{n\in\{0,\ldots,N\}}\nnrm{P^1(\tilde E(t_n)-E(t_n))}{\bo( \cH^{\alpha},\dot H^0)}\nnrm{B}{\bo(\dot H^{\alpha-1}, \cH^\alpha)}\\
&\leq C(T)\big(h^{\min(\alpha\frac{r}{r+1},r)}+(\dt)^{\min(\alpha\frac{p}{p+1},p)}\big)
\end{aligned}
\end{equation}
As the operator $P^1(\tilde E(t)-E(t))B\in\bo(\dot H^0)$ is symmetric in $\dot H^0$ and since $\dot H^{-\alpha+1}$ can be identified with the dual space of $\dot H^{\alpha-1}$, we have
\begin{equation*}
\nnrm{P^1(\tilde E(t)-E(t))B}{\bo(\dot H^{\alpha-1},\dot H^0)}=\nnrm{P^1(\tilde E(t)-E(t))B}{\bo(\dot H^0,\dot H^{-\alpha+1})}
\end{equation*}
and therefore also
\begin{equation}\label{eq:proofSWE7}
\sup_{n\in\{0,\ldots,N\}}\nnrm{P^1(\tilde E(t_n)-E(t_n))B}{\bo(\dot H^0,\dot H^{-\alpha+1})}\leq C(T)\big(h^{\min(\alpha\frac{r}{r+1},r)}+(\dt)^{\min(\alpha\frac{p}{p+1},p)}\big).
\end{equation}
Next, we use the fact that $\dot H^{(\alpha/2)-1}$ and $\dot H^{-\alpha/2}$ can be represented as the real interpolation spaces $(\dot H^0,\dot H^{\alpha-1})_{\theta,2}$ and $(\dot H^{-\alpha+1},\dot H^0)_{\theta,2}$, respectively, where $\theta=((\alpha/2)-1)/(\alpha-1)\in(0,1)$, cf.~Remark~\ref{rem:interpolation}. Thus, interpolation between \eqref{eq:proofSWE6} and \eqref{eq:proofSWE7} yields
\begin{equation*}
\begin{aligned}
&\sup_{n\in\{0,\ldots,N\}}\nnrm{P^1(\tilde E(t_n)-E(t_n))B}{\bo(\dot H^{(\alpha/2)-1},\dot H^{-\alpha/2})}\\
&\leq\sup_{n\in\{0,\ldots,N\}}C(\alpha)\nnrm{P^1(\tilde E(t_n)-E(t_n))B}{\bo(\dot H^0,\dot H^{-\alpha+1})}^{1-\theta}\nnrm{P^1(\tilde E(t_n)-E(t_n))B}{\bo(\dot H^{\alpha-1},\dot H^0)}^\theta\\
&\leq C(T,\alpha)\big(h^{\min(\alpha\frac{r}{r+1},r)}+(\dt)^{\min(\alpha\frac{p}{p+1},p)}\big),
\end{aligned}
\end{equation*}
see, e.g., Definition 1.2.2/2 and Theorem 1.3.3(a) in \cite{Tri78}.

For $0\leq\alpha\leq2$, we note that
\begin{equation*}
\begin{aligned}
\nnrm{P^1(\tilde E(t_n)-E(t_n))B}{\bo(\dot H^0,\dot H^{-1})}
&=\nnrm{P^1(\tilde E(t_n)-E(t_n))B}{\bo(\dot H^{1},\dot H^0)}\\
&\leq\nnrm{P^1(\tilde E(t_n)-E(t_n))}{\bo( \cH^2,\dot H^0)}\nnrm{B}{\bo(\dot H^1, \cH^2)},\\
\end{aligned}
\end{equation*}
where we used again the symmetry of $P^1(\tilde E(t)-E(t))B\in\bo(\dot H^0)$.
By \eqref{eq:detEstSWE1} we obtain
\begin{equation}\label{eq:proofSWE8}
\sup_{n\in\{0,\ldots,N\}}
\nnrm{P^1(\tilde E(t_n)-E(t_n))B}{\bo(\dot H^0,\dot H^{-1})}
\leq C(T)\big(h^{\min(2\frac{r}{r+1},r)}+(\dt)^{\min(2\frac{p}{p+1},p)}\big),
\end{equation}
which is \eqref{eq:detEstSWE5} for $\alpha=0$.
Moreover, also by \eqref{eq:detEstSWE1},
\begin{equation}\label{eq:proofSWE9}
\begin{aligned}
\sup_{n\in\{0,\ldots,N\}}&\nnrm{P^1(\tilde E(t_n)-E(t_n))B}{\bo(\dot H^{-1},\dot H^{0})}\\
&\leq \sup_{n\in\{0,\ldots,N\}}\nnrm{P^1(\tilde E(t_n)-E(t_n))}{\bo( H,\dot H^{0})}\nnrm{B}{\bo(\dot H^{-1}, H)}
\;\leq\; C(T),
\end{aligned}
\end{equation}
i.e., we have \eqref{eq:detEstSWE5} for $\alpha=2$. Finally, if $\alpha\in(0,2)$, interpolation with $\theta=(\alpha/2)-1\in(0,1)$ between \eqref{eq:proofSWE8} and \eqref{eq:proofSWE9} gives
\begin{align*}
    &\sup_{n\in\{0,\ldots,N\}}\nnrm{P^1(\tilde E(t_n)-E(t_n))B}{\bo(\dot H^{(\alpha/2)-1},\dot H^{-\alpha/2})}\\
    &\quad\leq\sup_{n\in\{0,\ldots,N\}}C(\alpha)\nnrm{P^1(\tilde E(t_n)-E(t_n))B}{\bo(\dot H^{0},\dot H^{-1})}^{1-\theta}\nnrm{P^1(\tilde E(t_n)-E(t_n))B}{\bo(\dot H^{-1},\dot H^{0})}^{\theta}\\
    &\quad\leq C(T,\alpha)\big(h^{\min(2\frac{r}{r+1},r)}+(\dt)^{\min(2\frac{p}{p+1},p)}\big)^{\frac\alpha2}\\
    &\quad= C(T,\alpha)\big(h^{2\frac{r}{r+1}}+(\dt)^{2\frac{p}{p+1}}\big)^{\frac\alpha2}\\
    &\quad\leq C(T,\alpha)\big(h^{\min(\alpha\frac{r}{r+1},r)}+(\dt)^{\min(\alpha\frac{p}{p+1},p)}\big).\qedhere
\end{align*}
\end{proof}

%

\begin{appendix}
\section{Poisson random measures and a comparison of\\ stochastic integrals}
\label{sec:PRM+SI}

Our proof of Theorem~\ref{thm:errorRep} is based on It\^{o}'s formula for Banach space-valued jump processes driven by Poisson random measures as presented in \cite{MRT13}. Alternatively, one could use It\^{o}'s formula 
as proved in \cite{GraPel74}, but the formula in \cite{MRT13} is more convenient in our setting.
In this section, we use Lemma~\ref{lem:integrandIsomorphism} to relate our setting to the setting in \cite{MRT13}.

It is well-known that the jumps of a L\'{e}vy process determine a Poisson random measure on the product space of the underlying time interval and the state space. We refer to \cite[Section~6]{PesZab07} for a definition and properties of Poisson random measures. For $(\omega,t)\in\Omega\times(0,\infty)$ we denote by $\Delta L(t)(\omega):=L(t)(\omega)-\lim_{s\nearrow t}L(s)(\omega)\in U_1$ the jump of a trajectory of $L$ at time $t$. Setting
\begin{equation*}
N(\omega):=\sum_{\Delta L(t)(\omega)\neq0}\delta_{(t,\Delta L(t)(\omega))},\quad \omega\in\Omega,
\end{equation*}
defines a Poisson random measure $N$ on $([0,\infty)\times U_1,\cB([0,\infty))\otimes\cB(U_1))$ with intensity measure (or compensator) $\lambda\otimes\nu$, where $\lambda$ is Lebesgue measure on $[0,\infty)$ and $\nu$ is the jump intensity measure of $L$.
This follows, e.g., from Theorem~6.5 in \cite{PesZab07} together with Theorems~4.9, 4.15, 4.23 and Lemma~4.25 therein.
We denote the compensated Poisson random measure by
\begin{equation}
q:=N-\lambda\otimes \nu.
\end{equation}

Let $V$ be a (real and separable) Hilbert space. The stochastic integral with respect to $q$ of functions in $L^2(\Omega_T\times U_1,\bP_T\otimes \nu;V)=L^2(\Omega_T\times U_1,\cP_T\otimes\cB(U_1),\bP_T\otimes \nu;V)$  is constructed as a linear isometry
\begin{equation*}
L^2(\Omega_T\times U_1,\bP_T\otimes\nu;V)\to \cM^2_T(V),\;
f\mapsto\Big(\int_0^t\int_{U_1} f(s,x)\,q(\dl s,\dl x)\Big)_{t\in[0,T]}.
\end{equation*}
In particular, the $V$-valued integral processes have c\`{a}dl\`{a}g modifications; we will always work with such a c\`{a}dl\`{a}g modification.
Using a standard stopping procedure, the stochastic integral can be extendend to functions $f\in L^0(\Omega_T\times U_1,\cP_T\otimes\cB(U_1),\bP_T\otimes \nu;V)$ such that
\[\bP\Big(\int_0^T\int_{U_1}\nnrm{f(s,x)}V^2\,\nu(\dl x)\,\dl s<\infty\Big)=1.\]
We refer to \cite{MRT13}, \cite{Pre10} and the references therein for details on stochastic integration w.r.t.\ Poisson random measures, compare also \cite[Section~8.7]{PesZab07}.

\begin{remark}
Stricty speaking, in \cite{MRT13} the integrands $f$ do not have to be predictable but only $\cF_t\otimes\cB(U_1)$-adapted and $\cF\otimes\cB([0,T])\otimes\cB(U_1)$-measurable. However, it is clear that in the case of predictable, i.e., $\cP_T\otimes\cB(U_1)$-measurable, and square integrable Hilbert space-valued integrands $f$ the stochastic integral in \cite{MRT13} coincides with the stochastic integral considered in \cite{PesZab07}, \cite{Pre10}. See \cite{RueTap12} for a detailed comparison of the different spaces of integrands.
\end{remark}

Since $\bE\int_0^T\int_{U_1}\nnrm{x}{U_1}^2\,\nu(\dl x)\dl t$ is finite for all $T<\infty$, it follows that the integral process $(\int_0^t\int_{U_1} x\,q(\dl s,\dl x))_{t\geq 0}$ is uniquely determined (up to indistinguishability) as a $U_1$-valued square-integrable
c\`{a}dl\`{a}g martingale. Taking into account the assumptions on the L\'{e}vy process $L$, the L\'{e}vy-Khinchin decomposition \cite[Theorem~4.23]{PesZab07}, the definition of $q$, and the construction of the stochastic integral w.r.t.\ $q$, it is not difficult to see that the processes $L$ and $(\int_0^{t}\int_{U_1} x\,q(\dl s,\dl x))_{t\geq0}$ are indistinguishable, i.e.,
\begin{equation}\label{eq:L=qInt}
\bP\Big(L(t)=\int_0^t\int_{U_1} x\,q(\dl s,\dl x)\quad\forall\; t\geq0\Big)=1.
\end{equation}

Using Lemma~\ref{lem:integrandIsomorphism}, we are now able to identify stochastic integrals w.r.t.\ $L$ and stochastic integrals w.r.t.\ the compensated Poisson random measure $q$. Recall from Remark~\ref{not:Phix} that we identify processes $\Phi\in L^2(\Omega_T,\bP_T;\hs(U_0,H))$ with the corresponding functions $\kappa(\Phi)\in L^2(\Omega_T\times U_1,\bP_T\otimes\nu;H)$. Thus, for such $\Phi$ the integral process $(\int_0^t\int_{U_1}\Phi(s)x\,q(\dl s,\dl x))_{t\in[0,T]}$ is defined.

\begin{lemma}\label{lem:comparisonIntegrals}
Given $\Phi\in L^2(\Omega_T,\bP_T;\hs(U_0,H))$, the $H$-valued c\`{a}dl\`{a}g integral processes $(\int_0^{t}\Phi(s)\,\dl L(s))_{t\in[0,T]}$ and $(\int_0^{t
}\int_{U_1} \Phi(s)x\,q(\dl s,\dl x))_{t\in[0,T]}$ are indistinguishable. That is, 
\[\bP\Big(\int_0^t\Phi(s)\,\dl L(s)=\int_0^t\int_{U_1}\Phi(s)x\,q(\dl s,\dl x)\quad \forall\,t\in [0,T]\Big)=1.\]
\end{lemma}

\begin{proof}
We first assume that $\Phi$ is a simple $\bo(U_1,H)$-valued process of the form
\begin{equation*}
\Phi(s)=\sum_{k=0}^{m-1}\one_{F_k}\one_{(t_k,t_{k+1}]}(s)\Phi_k,\quad s\in[0,T],
\end{equation*}
with $0\leq t_0<t_1<\cdots<t_m\leq T$, $m\in\bN$, $F_k\in\cF_{t_k}$ and $\Phi_k\in\bo(U_1,H)$. Recall from Section~\ref{sec:LSEE} that $\bo(U_1,H)$ is a subspace of $\hs(U_0,H)$. Using \eqref{eq:L=qInt} and applying standard arguments for the evaluation of stochastic integrals, we obtain for fixed $t\in[0,T]$, $\bP$-almost surely,
\begin{align*}
\int_0^t\Phi(s)\,\dl L(s)&=\sum_{k=0}^{m-1}\one_{F_k}\Phi_k(L(t_{k+1}\wedge t)-L(t_k\wedge t))\\
&=\sum_{k=0}^{m-1}\one_{F_k}\Phi_k\Big(\int_0^T\int_{U_1} \one_{(t_k\wedge t, t_{k+1},\wedge t]}(s) x\,q(\dl s,\dl x)\Big)\\
&=\sum_{k=0}^{m-1}\int_0^T\int_{U_1} \one_{F_k}\one_{(t_k\wedge t, t_{k+1},\wedge t]}(s)\Phi_k x\,q(\dl s,\dl x)\\
&=\int_0^t\int_{U_1}\Phi(s) x\,q(\dl s,\dl x).
\end{align*}
Since both processes are right-continuous, we see that the processes $(\int_0^{t}\Phi(s)\,\dl L(s))_{t\in[0,T]}$ and $(\int_0^{t
}\int_{U_1} \Phi(s)x\,q(\dl s,\dl x))_{t\in[0,T]}$ are indistinguishable.


For general $\Phi\in L^2(\Omega_T,\bP_T;\hs(U_0,H))$, take a sequence $(\Phi_n)_{n\in\bN}$ of simple $\bo(U_1,H)$-valued process
such that $\Phi_n\to\Phi$ in $L^2(\Omega_T,\bP_T;\hs(U_0,H))$; see, e.g., \cite[Proposition~2.3.8]{PreRoeck07} for a proof of the existence of such a sequence.
Then, the
processes $$\int_0^{\arg}\Phi_n(s)\,\dl L(s)=(\int_0^{t}\Phi_n(s)\,\dl L(s))_{t\in[0,T]}$$ and $\int_0^{\arg}\int_{U_1}\Phi_n(s)x\,q(\dl s,\dl x)=(\int_0^{t
}\int_{U_1} \Phi_n(s)x\,q(\dl s,\dl x))_{t\in[0,T]}$ are indistinguishable for all $n\in\bN$, and
we have the convergence $\int_0^{\arg}\Phi_n(s)\,\dl L(s)\to\int_0^{\arg}\Phi(s)\,\dl L(s)$ in $\cM^2_T(H)$ by the construction of the stochastic integral w.r.t.\ $L$.
 According to Lemma~\ref{lem:integrandIsomorphism}, the convergence $\Phi_n\to\Phi$ in $L^2(\Omega_T,\bP_T;\hs(U_0,H))$ entails the convergengence $\kappa(\Phi_n)\to \kappa(\Phi)$ in $L^2(\Omega_T\times U_1,\bP_T\otimes\nu;H)$, so that we also have $$\int_0^{\arg}\int_{U_1}\Phi_n(s)x\,q(\dl s,\dl x)\to\int_0^{\arg}\int_{U_1}\Phi(s)x\,q(\dl s,\dl x)$$ in $\cM^2_T(H)$. Thus, $\int_0^{\arg}\Phi(s)\,\dl L(s)=\int_0^{\arg}\int_{U_1}\Phi(s)x\,q(\dl s,\dl x)$ as an equality in $\cM^2_T(H)$, which yields the assertion.
\end{proof}

\end{appendix}

\providecommand{\bysame}{\leavevmode\hbox to3em{\hrulefill}\thinspace}

\section*{}
\noindent Mih\'{a}ly Kov\'{a}cs\\
Department of  Mathematics and Statistics\\
University of Otago \\
P.O.~Box 56, Dunedin, New Zealand\\
E-mail: mkovacs@maths.otago.ac.nz \\

\noindent Felix Lindner\\
Fachbereich Mathematik \\
Technische Universit\"{a}t Kaiserslautern\\
Postfach 3049, Kaiserslautern, Germany \\
E-mail: lindner@mathematik.uni-kl.de \\

\noindent Ren\'{e} L.~Schilling\\
Fachrichtung Mathematik\\
Institut f\"{u}r Mathematische Stochastik \\
Technische Universit\"{a}t Dresden \\
01062 Dresden, Germany \\
E-mail: rene.schilling@tu-dresden.de


\begin{thebibliography}{99}
\addcontentsline{toc}{chapter}{Bibliography}

\bibitem{AppRie10}
D.~Applebaum, M.~Riedle:
Cylindrical L\'{e}vy processes in Banach spaces.
{\it Proc.\ London Math.\ Soc.} {\bf 101}(3) (2010) 697--726.

\bibitem{ABHN11}
W.~Arendt, C.J.K.~Batty, M.~Hieber, F.~Neubrander:
{\it Vector-valued Laplace transforms and Cauchy problems.} Birkh\"{a}user, Basel 2011.

\bibitem{BGK15}
B. Baeumer, M. Geissert, M. Kov\'acs: Existence, uniqueness and regularity for a class of semilinear
 stochastic Volterra equations with multiplicative noise.
 {\it J. Differential Equations}  {\bf 258}(2)  (2015) 535--554.

\bibitem{BakBra79}
G.A.~Baker, J.H.~Bramble:
Semidiscrete and single step fully discrete approximations for second order hyperbolic equations.
{\it RAIRO Numer.\ Anal.} {\bf 13} (1979) 76--100.

\bibitem{Bar10}
A.~Barth:
A finite element method for martingale-driven stochastic partial differential equations.
{\it Commun.\ Stoch.\ Anal.} {\bf 4}(3) (2010) 355--375.

\bibitem{BarLan12a}
A.~Barth, A.~Lang:
Simulation of stochastic partial differential equations using finite element methods.
{\it Stochastics} {\bf 84}(2-3) (2012) 217--231.

\bibitem{BarLan12b}
A.~Barth, A.~Lang:
Milstein approximation for advection-diffusion equations driven by multiplicative noncontinuous martingale noises.
{\it Appl.\ Math.\ Optim.} {\bf 66} (2012) 387--413.

\bibitem{BreTho80}
P.~Brenner, V.~Thome\'{e}:
On rational approximations of groups of operators.
{\it SIAM J.~Numer.~Anal.} {\bf 17}(1) (1980) 119--125.


\bibitem{BrzZab10}
Z.~Brze{\'z}niak, J.~Zabczyk:
Regularity of Ornstein-Uhlenbeck processes driven by a L\'{e}vy white noise.
{\it Potential Anal.} {\bf 32} (2010) 153--188.

\bibitem{CDaPP} P. Cl\'ement, G. Da Prato, J. Pr\"uss: White noise perturbation of the equations of linear parabolic
viscoelasticity. Rend. Istit. Mat. Univ. Trieste  \textbf{XXIX} (1997) 207--220.

\bibitem{CJK14}
D.~Conus, A.~Jentzen, R.~Kurniawan:
Weak convergence rates of spectral Galerkin approximations for SPDEs with nonlinear diffusion coefficients.
{\it Preprint}, 2014. arXiv:1408.1108v1

\bibitem{DPJR12}
G.~Da~Prato, A.~Jentzen, M.~R\"{o}ckner:
A mild It\^{o} formula for SPDEs.
{\it Preprint}, 2012. arXiv:1009.3526v4

\bibitem{DPZ92}
G.~Da~Prato, J.~Zabczyk:
{\it Stochastic equations in infinite dimensions}.
Cambridge University Press, Cambridge 1992.

\bibitem{debussche_printems}
A.~Debussche, J.~Printems:
Weak order for the discretization of the stochastic heat equation.
{\it Math. Comp.} \textbf{78} (2009) 845--863.

\bibitem{DunHauPro12}
T.~Dunst, E.~Hausenblas, A.~Prohl:
Approximate Euler method for parabolic stochastic partial differential equations driven by space-time L\'{e}vy noise.
{\it SIAM J.~Numer.~Anal.} {\bf 50}(6) (2012) 2873--2896.


\bibitem{Fromm93} S. J. Fromm:  Potential space estimates for Green potentials in convex domains.
 {\it Proc. Amer. Math. Soc.}  {\bf 119}(1)  (1993) 225--233.


\bibitem{GraPel74}
J.B.~Gravereaux, J.~Pellaumail:
Formule de Ito pour des processus non continus \`{a} valeurs dans des espaces de Banach.
{\it Ann.\ Inst.\ Henri Poincar\'{e} Probab. Statist.} {\bf 10}(4) (1974) 399--422.

\bibitem{Hau08}
E.~Hausenblas:
Finite element approximation of stochastic partial differential equations driven by Poisson random measures of jump type.
{\it SIAM J.\ Numer.\ Anal.} {\bf 46} (2008) 437--471.

\bibitem{HauMar06}
E.~Hausenblas, I.~Marchis:
A numerical approximation of parabolic stochastic partial differential equations driven by a Poisson random measure.
{\it BIT} {\bf 46} (2006) 773--811.

\bibitem{JKMP05}
J.~Jacod, T.G.~Kurtz, S.~M\'{e}l\'{e}ard, P.~Protter:
The approximate Euler method for L\'{e}vy driven stochastic differential equations.
{\it Ann.\ Inst.\ Henri Poincar\'{e} Probab.\ Stat.} {\bf 41}(3) (2005) 523--558.

\bibitem{KovLarLin11}
M.~Kov\'{a}cs, S.~Larsson, F.~Lindgren:
Weak convergence of finite element approximations of linear stochastic evolution equations with additive noise.
{\it BIT} {\bf 52} (2012) 85--108.

\bibitem{KovLarLin13}
M.~Kov\'{a}cs, S.~Larsson, F.~Lindgren:
Weak convergence of finite element approximations of linear stochastic evolution equations with additive noise II: Fully discrete schemes.
{\it BIT} {\bf 53} (2013) 497-525.

\bibitem{KovPri14a}
M.~Kov\'{a}cs, J.~Printems:
Strong order of convergence of a fully discrete approximation of a linear stochastic Volterra type evolution equation.
{\it Math.~Comp.} {\bf 83}(298) (2014) 2325--2346.

\bibitem{KovPri14b}
M.~Kov\'{a}cs, J.~Printems:
Weak convergence of a fully discrete approximation of a linear stochastic evolution equation with a positive-type memory term.
{\it J.\ Math.\ Anal. Appl.} {\bf 413} (2014) 939–952.

\bibitem{KovLarSae10}
M.~Kov\'{a}cs, S.~Larsson, F.~Saedpanah:
Finite element approximation of the linear stochastic wave equation with additive noise.
{\it SIAM J.\ Numer.\ Anal.} {\bf 48} (2010) 408--427.

\bibitem{Lan10}
A.~Lang:
Mean square convergence of a semidiscrete scheme for SPDEs of Zakai type driven by square integrable martingales.
{\it Procedia Computer Science} {\bf 1} (2012) 1615--1623.

\bibitem{LarTho03}
S.~Larsson, V.~Thom\'{e}e:
{\it Partial differential equations with numerical methods.}
Springer, Berlin 2003.

\bibitem{LinSch13}
F.~Lindner, R.L.~Schilling:
Weak order for the discretization of the stochastic heat equation driven by impulsive noise.
{\it Potential Anal.} {\bf 38}(2) (2013) 345--379.

\bibitem{lubich88} C. Lubich: Convolution quadrature and discretized operational calculus. I. \textit{Numer. Math.} \textbf{52} (1988) 129--145.
\bibitem{lubich88II} C. Lubich: Convolution quadrature and discretized operational calculus. II. \textit{Numer. Math.} \textbf{52} (1988) 413--425.

\bibitem{Lubich_et_al1996} C. Lubich, I. Sloan and V. Thom\'ee: Nonsmooth data error estimates for approximations of an evolution equation with a positive-type memory term. \textit{Math. Comp.} \textbf{65} (1996) 1--17.

\bibitem{MRT13}
V.~Mandrekar, B.~R\"{u}diger, S.~Tappe: It\^{o}'s formula for Banach space-valued jump processes driven by Poisson random measures. In: C.~Dalang, M.~Dozzi, F.~Russo (eds.): {\it Seminar on stochastic analysis, random fields and applications VII. Centro Stefano Franscini, Ascona, May 2011.} Progress in Probability {\bf 67}. Birkh\"{a}user, Basel 2013, 171--186.

\bibitem{MikZha11}
R.~Mikulevi\v{c}ius, C.~Zhang:
On the rate of convergence of weak Euler approximation for nondegenerate SDEs driven by L\'{e}vy processes.
{\it Stochastic Process.\ Appl.} {\bf 121} (2011) 1720--1748.


\bibitem{Met82}
M.~M\'{e}tivier:
{\it Semimartingales. A course on stochastic processes.}
de Gruyter, Berlin 1982.

\bibitem{MetPel80}
M.~M\'{e}tivier, J.~Pellaumail:
{\it Stochastic integration.}
Academic Press, New York 1980.

\bibitem{NorWan79}
S.P.~Norsett, G.~Wanner:
The real pole sandwich for rational approximations and oscillation equations.
{\it BIT} {\bf 19} (1979) 79--94.

\bibitem{Paz83}
A.~Pazy:
{\it Semigroups of linear operators and applications to partial differential equations.}
Springer, New York 1983.

\bibitem{PesZab07}
S.~Peszat, J.~Zabczyk:
{\it Stochastic partial differential equations with L\'evy noise. An evolution equation approach.}
Cambridge University Press, Cambridge 2007.

\bibitem{PlaBru10}
E.~Platen, N.~Bruti-Liberati:
{\it Numerical solutions of stochastic differential equations with jumps in finance.}
Springer,
Berlin 2010.

\bibitem{Pre10}
C.~Pr\'{e}v\^{o}t:
Existence, uniqueness and regularity w.r.t.\ the initial condition of mild solutions of SPDEs driven by Poisson noise.
{\it Infin.\ Dimens.\ Anal.\ Quantum Probab.\ Relat.\ Top.} {\bf 13}(1) (2010) 133--163.

\bibitem{PreRoeck07}
C.~Pr\'{e}v\^{o}t, M. R\"ockner:
{\it A concise course on stochastic partial differential equations.}
Springer,
Lecture Notes in Mathematics {\bf 1905},
Berlin 2007.

\bibitem{ProTal97}
P.~Protter, D.~Talay:
The Euler scheme for L\'{e}vy driven stochastic differential equations.
{\it Ann. Probab.} {\bf 25}(1) (1997) 393--423.

\bibitem{Rie12}
M.\ Riedle:
Stochastic integration with respect to cylindrical L\'{e}vy processes in Hilbert spaces: an $L^2$ approach.
{\it Infin.\ Dimens.\ Anal.\ Quantum.\ Probab.\ Relat.\ Top.}~{\bf 17}(1) (2014) 1450008.

\bibitem{Rie14}
M.\ Riedle:
Ornstein-Uhlenbeck processes driven by cylindrical L\'{e}vy processes.
{\it Preprint}, 2014.  arXiv:1212.3832v3


\bibitem{RueTap12}
B.~R\"{u}diger, S.~Tappe: Isomorphisms for spaces of predictable processes and an extension of the It\^{o} integral.
{\it Stoch.\ Anal.\ Appl.} {\bf 30}(3) (2012) 529--537.

\bibitem{Sato13}
K.~Sato:
{\it L\'{e}vy processes and infinitely divisible distributions.} 
Cambridge University Press,
Cambridge 2013.

\bibitem{SSV10}
R.L.~Schilling, R.~Song, Z.~Vondracek:
{\it Bernstein functions. Theory and applications.}
de Gruyter,
Berlin 2010.

\bibitem{Sperlich} S. Sperlich: On parabolic Volterra equations disturbed by fractional Brownian
 motions. \textit{Stoch. Anal. Appl.}  \textbf{27}(1) (2009) 74--94.

\bibitem{Tho06}
V.~Thom\'{e}e:
{\it Galerkin finite element methods for parabolic problems.} (2nd edn) Springer,
Springer Series in Computational Mathematics {\bf 25},
Berlin 2006.

\bibitem{Tri78}
H.~Triebel:
{\it Interpolation theory, function spaces, differential operators.} (2nd edn)
Johann Ambrosius Barth, Heidelberg 1995 .

\bibitem{Wei80}
J.~Weidmann:
{\it Linear operators in Hilbert spaces.}
Springer,
Graduate texts in mathematics {\bf 68},
 New York 1980.

\bibitem{Yan05}
Y.~Yan:
Galerkin finite element methods for stochastic parabolic partial differential equations.
{\it SIAM J.~Numer.~Anal.} {\bf 43}(4) (2005) 1363--1384.

\end{thebibliography}
\end{document}